\documentclass{article}
\usepackage{graphicx} % Required for inserting images

\usepackage{setspace}

\usepackage{amsmath}
\usepackage{amssymb}
\usepackage{amsthm}
\usepackage{mathtools}
\usepackage{empheq}
\usepackage{enumitem}
\usepackage{titlesec}
\usepackage{graphicx}
\usepackage{caption}
\usepackage{subcaption}
\usepackage{diagbox}
\usepackage{hyperref}

\usepackage[T1]{fontenc}
\usepackage[utf8]{inputenc}
\usepackage{comment}
\usepackage{indentfirst}
\usepackage[top=1.25in,left=1.25in,right=1.25in]{geometry}

\numberwithin{equation}{section}

\title{\Large{\uppercase{\bf Deza graphs and regular polyhedra}}}
\author{\Large{Riccardo W. Maffucci}}
\date{}
\newcommand{\Addresses}{  
		R.W.~Maffucci, \textsc{Dipartimento di Matematica, Universit\`a di Torino\\\indent Via Carlo Alberto 10, Turin 10123, Italy}\par\nopagebreak\vspace{-0.35cm}
		\textit{E-mail address}, R.W.~Maffucci: \href{mailto:riccardowm@hotmail.com}{\texttt{riccardowm@hotmail.com}}
  }

\def\ct{\mathcal{T}}
\def\calc{\mathcal{C}}
\def\cl{\mathcal{L}}
\def\cm{\mathcal{M}}
\def\cf{\mathcal{F}}
\def\calr{\mathcal{R}}

\newtheorem{thm}{Theorem}[section]
\newtheorem{lemma}[thm]{Lemma}
\newtheorem{prop}[thm]{Proposition}

\newtheorem{cor}[thm]{Corollary}
\newtheorem{defin}[thm]{Definition}

\begin{document}
\titleformat{\section}
  {\Large\scshape}{\thesection}{1em}{}
\titleformat{\subsection}
  {\large\scshape}{\thesubsection}{1em}{}
\maketitle
\Addresses

\begin{abstract}
We classify all regular polyhedra according to their type i.e., the collection of numbers of common neighbours that any pair of distinct vertices may have (polyhedra are planar, $3$-connected graphs). As an application, we recover the classification of planar Deza graphs.

Next, we focus on the class of quartic polyhedral Deza graphs, and completely characterise it in terms of medial graphs of certain specific cubic polyhedra. Furthermore, within the aforementioned class of quartic polyhedral Deza graphs, we study the extremal graphs with respect to the ratio of number of triangular faces to the total. In the maximal extreme, these notably coincide with the class of line graphs of cubic polyhedra of girth $5$.

We also fully characterise the quartic polyhedra of type $\{0,1,2,3\}$, and in particular we prove that none of them are medial graphs.

On one hand our findings fit within the novel research area of common neighbours in graphs. On the other hand, our findings imply general properties of regular planar graphs and regular polyhedra.
\end{abstract}

\section{Introduction}
\subsection{Planar Deza graphs}
The graphs that appear in this paper are simple, undirected, and finite. We will write $V(G)$ and $E(G)$ for the vertex and edge sets of $G$. A graph is $r$-regular if every vertex has degree $r$, that is to say, has exactly $r$ neighbours (adjacent vertices). A graph is $k$-connected if it has more than $k$ vertices, and however one removes fewer than $k$ vertices, the resulting graph is connected. We say that a graph has connectivity $k$ if it is $k$-connected, but not $k+1$-connected.

A regular graph $G$ is called a Deza graph of type $\{\lambda,\mu\}$, or a $\{\lambda,\mu\}$-Deza graph, if there exist non-negative integers $\lambda\leq\mu$ such that for every distinct $u,v\in V(G)$, the number of common neighbours of $u,v$ is either $\lambda$ or $\mu$. Deza graphs are a generalisation of strongly regular graphs proposed in \cite{erickson1999deza} (see also \cite{deza1994ridge}). In a strongly regular graph, there exist $\lambda_1,\lambda_2$ such that every pair of adjacent vertices has $\lambda_1$ common neighbours, and every pair of non-adjacent vertices has $\lambda_2$ common neighbours. Substantial attention has since been given to the study of Deza graphs \cite{gavrilyuk2014vertex,kabanov2019strictly,goryainov2019deza,kabanov2020deza,goryainov2021deza} and related topics \cite{gengsheng2003directed,wang2006deza,haemers2011divisible,jia2018generalized,gavrilyuk2024strongly,crnkovic2025q}.

In this paper, we focus on $r$-regular planar graphs. Here we have $r\leq 5$. The terms cubic, quartic, and quintic refer to $3$-, $4$-, and $5$-regular graphs respectively. Our starting point is the following classification of planar Deza graphs, implicit in the literature.

\begin{prop}[{cf.\ \cite{limaye2005regular,brouwer2006classification,goryainov2021enumeration}}]
	\label{prop:1}
	All planar Deza graphs are listed in Tables \ref{t:1} ($3$-connected) and \ref{t:2} (not $3$-connected).
	\\
	Specifically a planar, $3$-connected graph $G$ is a Deza graph if and only if either $G$ is $3$-regular with no quadrangular faces, or $G$ is $4$-regular with no $4$-cycles, or $G$ is one of the five graphs in Figure \ref{fig:s}. 
\end{prop}
In Appendix \ref{app:a}, we will revisit Proposition \ref{prop:1}, and present a proof that makes use of the fine properties of planar, $3$-connected graphs i.e., \textbf{polyhedral graphs}, investigated in this paper.
\begin{figure}[ht]
	\centering
	\begin{subfigure}{0.19\textwidth}
		\centering
		\includegraphics[width=2.cm]{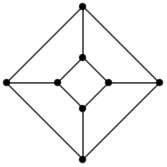}
		\caption{The cube.}
		\label{fig:s1}
	\end{subfigure}
	\hfill
	\begin{subfigure}{0.57\textwidth}
		\centering
		\includegraphics[width=2.cm]{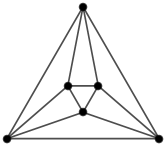}
		\hspace{0.025cm}
		\includegraphics[width=2.cm]{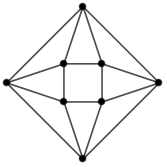}
		\hspace{0.05cm}
		\includegraphics[width=2.cm]{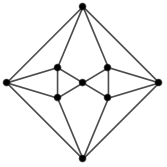}
		\caption{The $4$-regular, planar graphs on up to $9$ vertices.}
		\label{fig:s4}
	\end{subfigure}
	\hfill
	\begin{subfigure}{0.19\textwidth}
		\centering
		\includegraphics[width=2.25cm]{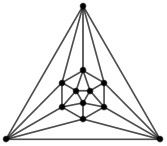}
		\caption{The icosahedron.}
		\label{fig:s5}
	\end{subfigure}
	\caption{The five exceptional polyhedral Deza graphs.}
	\label{fig:s}
\end{figure}

\begin{table}[ht]
	\centering
	$\begin{array}{|c|c|c|}
		\hline \text{Polyhedron }G &\text{regularity}&\text{type of Deza graph}\\
		\hline \text{tetrahedron}&3&\{2\}\\
		\hline \text{cube}&3&\{0,2\}\\
		\hline \text{cubic, no quadrangular faces, even }|V(G)|\geq 10&3&\{0,1\}\\
		\hline \text{octahedron}&4&\{2,4\}\\
		\hline \text{square antiprism -- Figure \ref{fig:s4}, centre}&4&\{1,2\}\\
		\hline \text{Figure \ref{fig:s4}, right}&4&\{1,2\}\\
		\hline \text{quartic, no $4$-cycles, }|V(G)|\geq 30&4&\{0,1\}\\
		\hline \text{icosahedron}&5&\{0,2\}\\
		\hline
	\end{array}$
	\caption{The planar Deza graphs of connectivity at least $3$.}
	\label{t:1}
\end{table}

\begin{table}[ht]
	\centering
	$\begin{array}{|c|c|c|c|}
		\hline \text{Graph}&\text{regularity}&\text{connectivity}&\text{type of Deza graph}\\
		\hline K_1\cup\dots\cup K_1&0&0&\{0\}\\
		\hline K_2\cup\dots\cup K_2&1&0 \text{ or }1&\{0\}\\
		\hline K_3&2&2&\{1\}\\
		\hline C_{i_1}\cup\dots\cup C_{i_\ell},\ \forall j: i_j\neq 4&2&0 \text{ or }2&\{0,1\}\\
		\hline C_{4}\cup\dots\cup C_{4}&2&0 \text{ or }2&\{0,2\}\\
		\hline \text{disjoint union of tetrahedra and/or cubes}&3&0&\{0,2\}\\
		\hline \text{disjoint union of icosahedra}&5&0&\{0,2\}\\
		\hline \text{cubic, no $4$-cycles}&3&0,1 \text{ or }2&\{0,1\}\\
		\hline \text{quartic, no $4$-cycles}&4&0,1 \text{ or }2&\{0,1\}\\
		%\hline \text{quintic, no $4$-cycles}&5&0,1 \text{ or }2&\{0,1\}\\
		\hline
	\end{array}$
	\caption{The planar Deza graphs of connectivity at most $2$.}
	\label{t:2}
\end{table}

%The main difficulties, and most interesting classes of solutions, are in the case of connectivity at least $3$.  
A graph is (the $1$-skeleton i.e., the wireframe of) a polyhedral solid if and only if it is planar and $3$-connected, a result due to Rademacher and Steinitz \cite{steinitz2013vorlesungen}. Thus we will refer to these graphs simply as `polyhedra'. They are also called $3$-polytopes in the literature. The study of regular polyhedra has fascinated scholars for millennia, and it was already known to the ancient Greeks that there are five vertex-regular and face-regular polyhedra, namely the Platonic solids. Interestingly, this result may be proven using only graph theory \cite{hahimo}.

%By analysing the planar, $3$-connected Deza graphs first, we may make use of the fine properties of the class of polyhedral graphs.
A planar graph is $2$-connected if and only if it has a planar immersion such that every region is delimited by a cycle (polygon). A planar, $2$-connected graph is $3$-connected if and only if every pair of distinct regions intersects in either the empty set, or a vertex, or an edge. Equivalently, a planar graph is $3$-connected if and only if each region is bounded by a cycle, and if two distinct regions have two common vertices, then these vertices are adjacent \cite[Section 4]{dieste}. We will call the regions of a polyhedron `faces', since they correspond in a natural way to the faces of the relevant polyhedral solid. We will denote by $[u_1,u_2,\dots,u_n]$ the face bounded by the cycle $u_1,u_2,\dots,u_n$.

Another way to identify a polyhedral solid with a graph is to define a vertex for each polyhedral face, and edges between adjacent faces of the solid. This defines the {\bf dual polyhedron} of the original one. The planar dual graph $G^*$ is defined more generally for any plane graph $G$, however it depends on the plane immersion of $G$, and $G^*$ is not necessarily a simple graph, as it may have multiple edges and/or loops. Notably, another fine property of the class of polyhedra is that the dual is unique, is always a simple graph, and is always another polyhedron. For instance, the duals of the cubic polyhedra are the triangulations of the sphere (maximal planar graphs). Indeed, a polyhedron has a unique immersion in the sphere (and hence a unique immersion in the plane, once an external region has been chosen), an observation due to Whitney \cite{whit32}. Thus it makes sense to talk about \textit{the} planar immersion of a polyhedron. Several recent papers deal with characterising and/or constructing a specific class of polyhedral graphs \cite{eppstein2021polyhedral,cui2021tight,costalonga2021constructing,mafpo1,mafpo2,zamfirescu2023hamiltonicity,gaspoz2024independence,delmaf,maffucci2024classification,maffucci2024rao,maffucci2024characterising,hollowbread2025generation}, and a few deal specifically with regular polyhedra \cite{brin05,hasheminezhad2011recursive, mafpo3,de2024cancellation,lo2025shortness}.

For the class of polyhedra, the property of being a Deza graph determines a subclass of regular polyhedral graphs that are in this sense `the most regular'. Note that the Platonic solids are all Deza graphs -- the tetrahedron and dodecahedron also verify the conditions of Proposition \ref{prop:1}, as they are cubic with no quadrangular faces.

We point out that requiring a $4$-regular polyhedron to have no $4$-cycles is stronger than requiring it to have no quadrangular faces. For instance, the pentagonal antiprism has no quadrangular faces but it contains $4$-cycles (so that by Proposition \ref{prop:1} the pentagonal antiprism is not a Deza graph). 

%The planar Deza graphs of connectivity at least $3$ along with their orders (number of vertices) and types are summarised in Table \ref{t:1}. The tetrahedron and octahedron are strongly regular. The lower bound $p\geq 30$ in the penultimate row will also be proven.

\subsection{Main results}
%on $4$-regular polyhedral Deza graphs, and an open problem}
Henceforth for a graph $G$ we will use the notations
\[N(u):=\{\text{neighbours of } u\},
\]
\[N(u,v):=\{\text{common neighbours of } u,v\},
\]
where $u\neq v\in V(G)$, and
\[A=A(G):=\{i : \exists u\neq v\in V(G) \text{ with } |N(u,v)|=i\},\]
suppressing the dependence on $G$ when there is no risk of confusion. We will say that $G$ is `of type $A$'. In particular, a regular graph $G$ is a Deza graph if and only if $|A|\leq 2$.

Our first main result is a stronger version of Proposition \ref{prop:1} for polyhedra, classifying all regular polyhedra according to their type.
\begin{thm}
	\label{thm:35r}
	Let $G$ be an $r$-regular polyhedron other than the tetrahedron and those in Figure \ref{fig:s}. If $r=3$, then either $A=\{0,1,2\}$ or $A=\{0,1\}$, depending on whether $G$ contains a quadrangular face or not.
	\\
	If $r=4$, then $A=\{0,1\}$ if and only if $G$ does not contain a $4$-cycle, else $A=\{0,1,2\}$ if and only if $G$ does not contain a subgraph isomorphic to $K(2,3)$, else $A=\{0,1,2,3\}$.
	\\
	If $r=5$, then
	$A\in\{\{0,1,2\}, \{0,1,2,3\}, \{0,1,2,4\}, \{0,1,2,3,4\}\}$.
\end{thm}
Theorem \ref{thm:35r} will be proven in Section \ref{sec:pf}. It will be the main ingredient in our proof of Proposition \ref{prop:1} (Appendix \ref{app:a}).

There are thus two main classes of polyhedral Deza graphs. One is given by the well-understood and studied class of cubic polyhedra with no quadrangular faces (duals of maximal planar graphs with no vertices of degree $4$). The other class is given by the quartic polyhedra with no $4$-cycles. These are more intriguing and less well-understood. We will uncover their fine structural properties next.

\begin{thm}
	\label{thm:4r}
	Let $G$ be a $4$-regular polyhedron other than those in Figure \ref{fig:s4}. Then $G$ is a Deza graph if and only if $G$ is the medial graph of $H$, where $H$ is a polyhedron with no vertices of degree $4$, no quadrangular faces, and no vertices of degree $3$ lying on a triangular face.
\end{thm}

Theorem \ref{thm:4r} will be proven in Section \ref{sec:4r}. The {\bf medial graph} of a polyhedron $\Gamma$ is defined by
\[V(\cm(\Gamma))=E(\Gamma)\]
and
\[E(\cm(\Gamma))=\{(e,e') : e,e' \text{ are consecutive edges on a face of $\Gamma$}\}.\]
The medial graph of a polyhedron is itself a polyhedron \cite[Lemma 2.1]{arcric}.

By definition, the medial graph is a spanning subgraph of the line graph. For any graph $\Gamma$, its line graph is defined as
\[V(\cl(\Gamma))=E(\Gamma)\]
and
\[E(\cl(\Gamma))=\{(e,e') : e,e' \text{ are incident edges in $\Gamma$}\}.\]
A line graph is not necessarily planar. A necessary and sufficient condition for the planarity of the line graph was found by Sedl{\'a}{\v{c}}ek \cite{sedlacek1964some, sedlavcek1990generalized}. All polyhedral line graphs were classified and constructed in \cite{hollowbread2025generation}. For cubic polyhedra, the line graph and medial graph coincide, and the resulting graph is a quartic polyhedron \cite{hollowbread2025generation}.

It is worth noting that the medial graph is the dual of the radial graph, also called the vertex-face graph. The radial graph is defined by
\[V(\calr(\Gamma))=V(\Gamma)\cup \cf(\Gamma)\]
with $\cf$ the set of faces, and
\[E(\calr(\Gamma))=\{vf : v\in V(\Gamma), \ f\in \cf(\Gamma), \ v\text{ lies on }f\text{ in }\Gamma\}.\]
We also record that
\[\calr(\Gamma)\simeq\calr(\Gamma^*) \quad\text{and}\quad \cm(\Gamma)\simeq\cm(\Gamma^*),\]
where $\Gamma^*$ is the dual polyhedron of $\Gamma$ \cite{mafpo3}.

The dual of a $4$-regular polyhedron is a $3$-connected quadrangulation of the sphere. These were constructed in \cite[Theorem 3]{brin05}. Thanks to this construction, we will be able to prove Theorem \ref{thm:4r}.

%\subsection{Extremal results, and an open problem}
We now consider the following problem in extremal graph theory. Let $G$ be a $4$-regular polyhedral Deza graph other than those in Figure \ref{fig:s4}, $f_3(G)$ the number of triangular faces, and $f(G)$ the total number of faces. We will prove that
\begin{equation*}
	\frac{f(G)}{2}+4\leq f_3(G)\leq\frac{2f(G)}{3}-\frac{4}{3}.
\end{equation*}
We wish to describe the subclasses of quartic polyhedral Deza graphs that achieve these bounds. For the upper bound, we have the following fine characterisation.
\begin{thm}
	\label{thm:4max}
	Let $G$ be a $4$-regular polyhedron other than those in Figure \ref{fig:s4}. Then $G$ is a Deza graph with maximal number of triangular faces with respect to the total if and only if $G$ is the line graph of a cubic polyhedron of girth $5$.
\end{thm}
Theorem \ref{thm:4max} will be proven in Section \ref{sec:4max}. The girth of a graph is the length of its shortest cycle. Hence in the statement of Theorem \ref{thm:4max}, we could have replaced the words `of girth $5$' with `of minimal face length $5$'. Furthermore, since we are dealing with cubic polyhedra, we could have replaced the words `line graph' with the words `medial graph'.

For the other extreme to Theorem \ref{thm:4max}, the characterisation is somewhat less clear.
\begin{prop}
	\label{prop:4min}
	Let $G$ be a $4$-regular polyhedron other than those in Figure \ref{fig:s4}. Then $G$ is a Deza graph with minimal number of triangular faces with respect to the total if and only if $G$ is the medial graph of a polyhedron $H$ satisfying the following conditions. Calling $p,q,f$ the number of vertices, edges, and faces of $H$, and $p_i,f_i$ the number of vertices of degree $i$ and faces of length $i$ respectively, then
	\[p=p_3+p_5, \qquad f=f_3+f_5, \qquad p_3+f_3=\frac{q}{2}+5,\]
	and each vertex of degree $3$ lies on three pentagonal faces.
\end{prop}
Proposition \ref{prop:4min} will be proven in Section \ref{sec:4min}. It would be nice to have a more precise characterisation for the $4$-regular polyhedral Deza graphs with minimal number of triangular faces.

Our next result concerns $4$-regular polyhedra of type $\{0,1,2,3\}$. To state it we need a preliminary definition.
\begin{defin}
	\label{def:1}
Let $G_1$ be a $4$-regular polyhedron, and $v=v(G_1),u=u(G_1),w=w(G_1)$ be consecutive vertices in this order on a face of $G_1$. We define
\begin{equation}
	\label{eq:prime}
G_1'=G_1'(u,v,w):=G_1+x+y+z+xu+yz+zu+xy+xv+yw-uv-uw,
\end{equation}
where $x=x(G_1)$, $y=y(G_1)$, and $z=z(G_1)$ are three new vertices. Moreover, for $G_2$ another $4$-regular polyhedron, we define
\[\ct(G_1,G_2)\]
(suppressing the dependency on $u(G_1),v(G_1),w(G_1),u(G_2),v(G_2),w(G_2)$) by taking $G_1'$ and $G_2'$ and then identifying $x(G_1)$ with $z(G_2)$, $y(G_1)$ with $y(G_2)$, and $z(G_1)$ with $x(G_2)$.
\end{defin}

The construction is illustrated in Figure \ref{fig:0123t}. The resulting $G:=\ct(G_1,G_2)$ is a planar, $4$-regular graph. It is easy to check that $G$ is $3$-connected. It satisfies
\[N(x(G_1),x(G_2))=\{y,u(G_1),u(G_2)\},\]
where $y=y(G_1)=y(G_2)$, thus $3\in A(G)$. By Theorem \ref{thm:35r}, $A(G)=\{0,1,2,3\}$. In the other direction, we will prove that all $4$-regular polyhedra of type $A(G)=\{0,1,2,3\}$ that do not contain a square pyramid are constructed as in Definition \ref{def:1}.
\begin{figure}[ht]
	\centering
	\includegraphics[width=10.cm]{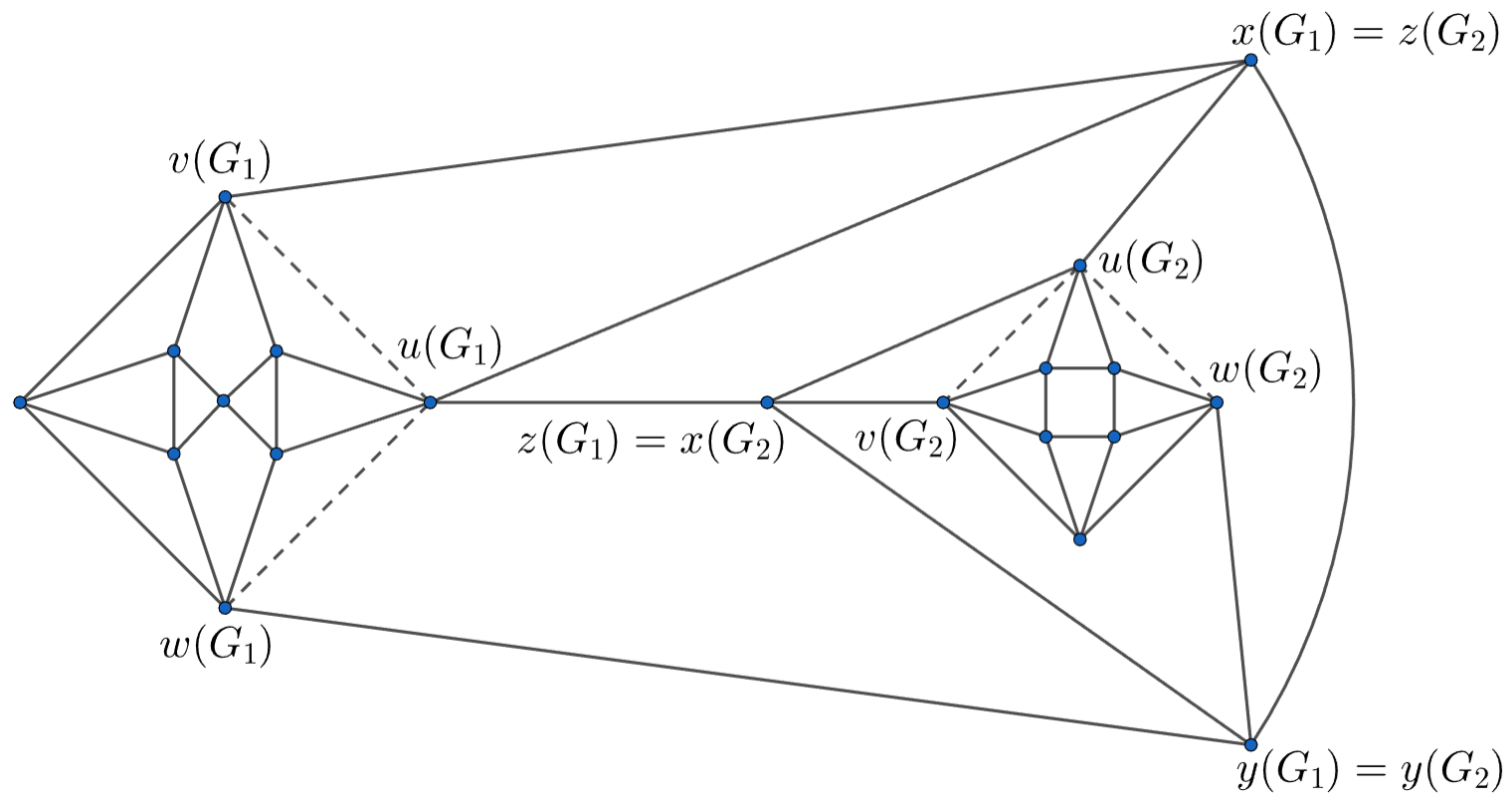}
	\caption{Illustration of $\ct(G_1,G_2)$, where $G_1,G_2$ are the third and second graphs in Figure \ref{fig:s4}. Dashed edges are deleted.}
	\label{fig:0123t}
\end{figure}

\begin{prop}
	\label{prop:0123}
Let $G$ be a $4$-regular polyhedron satisfying \[A(G)=\{0,1,2,3\}.\]
Then either $G$ contains a square pyramid as subgraph, or there exist $4$-regular polyhedra $G_1,G_2$ such that
\[G\simeq\ct(G_1,G_2).\]
\\
Furthermore, $G$ is not the medial graph of a polyhedron.
\end{prop}
The proof of Proposition \ref{prop:0123} may be found in Section \ref{sec:0123}. To summarise, $4$-regular polyhedra of type $\{0,1\}$ are always medial graphs of polyhedra (as stated in Theorem \ref{thm:4r}), those of type $\{0,1,2\}$ may or may not be medial graphs of polyhedra, whereas those of type $\{0,1,2,3\}$ are never medial graphs of polyhedra (as stated in Proposition \ref{prop:0123}).

\subsection{Context and discussion}
\label{sec:disc}
In \cite{maffucci2025classification} a classification of all polyhedral graphs according to their set $A$ was found. A polyhedron is of type $\{0,1\}$ if and only if it does not contain $4$-cycles. There are vast classes of polyhedra for each of the types $\{0,1,2\}$, $\{0,1,2,3\}$, $\{0,1,2,4\}$, $\{0,1,2,3,4\}$. In the present paper we classify regular polyhedra according to their type, and focus on their fine properties.

In \cite{maffucci2025common}, we generalised the discussion to the set $A_n(G)$, $n\geq 1$ of {\bf quantities of common neighbours for every $n$-tuple of distinct vertices} in a graph $G$. The set $A$ of this paper is accordingly denoted by $A_2$ in \cite{maffucci2025common}. In \cite{maffucci2025common}, all planar graphs were classified according to their set $A_n$, for every $n\geq 1$. This included the completion of the classification of planar graphs according to $A_2$, that was the object of \cite{maffucci2025classification} in the case of $3$-connected planar graphs.

The set $A_n$ may be considered as a version without multiplicities of the $n$\textbf{-degree sequence} of a graph, that lists quantities of common neighbours for every $n$-tuple of distinct vertices, in weakly decreasing order (the $1$-degree sequence is the usual degree sequence). This is a new and promising area of research. We refer the interested reader to \cite[\S 1.4]{maffucci2025classification}.

Furthermore, the results of this paper imply the following general properties of regular planar graphs and regular polyhedra.
\begin{cor}
	\label{cor:1}
For $r\in\{3,4,5\}$, $r$-regular polyhedra do not contain subgraphs isomorphic to $K(2,r)$.
\\
Polyhedral medial graphs (subclass of the $4$-regular polyhedra) do not contain subgraphs isomorphic to $K(2,3)$.
\\
Every $5$-regular planar graph contains a $4$-cycle.
\end{cor}
The three statements in Corollary \ref{cor:1} follow readily from Theorem \ref{thm:35r}, Proposition \ref{prop:0123}, and Proposition \ref{prop:30} (see Section \ref{sec:pf}) respectively.

\paragraph{Plan of the paper.}
The rest of this paper is dedicated to proving the stated results. Section \ref{sec:pf} contains the proof of Theorem \ref{thm:35r} classifying all types of regular polyhedra. Section \ref{sec:4} deals with the proofs of Theorems \ref{thm:4r} and \ref{thm:4max}, and of Proposition \ref{prop:4min}, concerning $4$-regular polyhedral Deza graphs. Section \ref{sec:0123} is dedicated to proving Proposition \ref{prop:0123} on $4$-regular polyhedra of type $\{0,1,2,3\}$.
\\
In Appendix \ref{app:a} we revisit Proposition \ref{prop:1} and give a proof that relies on the fine properties of polyhedral graphs. Certain technical details in the proof of Theorem \ref{thm:35r} to rule out finitely many potentially exceptional cases are relegated to Appendix \ref{app:b}. Lastly, in Appendix \ref{app:c} one may find a self-contained proof of a result that is going to be stated and applied in Section \ref{sec:pf}, and that follows from \cite[Theorem 1.3]{maffucci2025classification}.

%\subsection{Further results}
%\begin{prop}
%	\label{prop:4r}
%	Let $G$ be a $4$-regular polyhedron. Then $G$ is a Deza graph and the number of triangular faces of $G$ is minimised if and only if $G$ is the graph in Figure .
%\end{prop}
%With the following result, we construct a wide class of $4$-regular, planar Deza graphs.

\section{The types of regular polyhedra: proof of Theorem \ref{thm:35r}}
\label{sec:pf}
If $G=(V,E)$ is a connected graph, and $S\subseteq V(G)$ is such that deleting the elements of $S$ from $G$ results in a disconnected graph, then $S$ is called a cutset, or $|S|$-cut of $G$. If $S=\{v\}$, then $v$ is a separating vertex. We begin with a basic lemma that will be of substantial use.
\begin{lemma}
	\label{le:sc}
Let $G$ be a polyhedron, $\calc$ a cycle of $G$, and
\[S=S(\calc):=\{v\in\calc : v\text{ has at least one neighbour inside of } \calc\}.\]
Then either $S=\emptyset$ or $|S|\geq 3$.
\end{lemma}
\begin{proof}
By contradiction if $1\leq|S|\leq 2$, then $S$ is a cutset, hence $G$ is not $3$-connected.
%If $S\neq\emptyset$, then either $S$ is a cutset, or $\calc$ is the boundary of the external face of $G$, and every vertex on $\calc$ belongs to $S$. If $S$ is a cutset, then as $G$ is $3$-connected, we have $|S|\geq 3$. If $\calc$ is the boundary of the external face of $G$, and every vertex on $\calc$ belongs to $S$, then again $|S|\geq 3$.
\end{proof}

Lemma \ref{le:sc} allows us to start making considerations about $A(G)$ for regular polyhedra $G$.
\begin{lemma}
	\label{le:0123}
If $G$ is an $r$-regular polyhedron other than the octahedron, then $A\subseteq\{0,1,\dots,r-1\}$.
\end{lemma}
\begin{proof}
Clearly graphs of maximum degree $r$, including $r$-regular graphs, satisfy $A\subseteq\{0,1,\dots,r\}$. Let $G$ be an $r$-regular polyhedron, $u\neq v\in V(G)$, and \[N(u,v)=\{a_1,a_2,\dots,a_r\},\]
with vertices appearing in this cyclic order around $u$ in the planar immersion of $G$, as in Figure \ref{fig:pf01}.
\begin{figure}[ht]
	\centering
	\includegraphics[width=3.5cm]{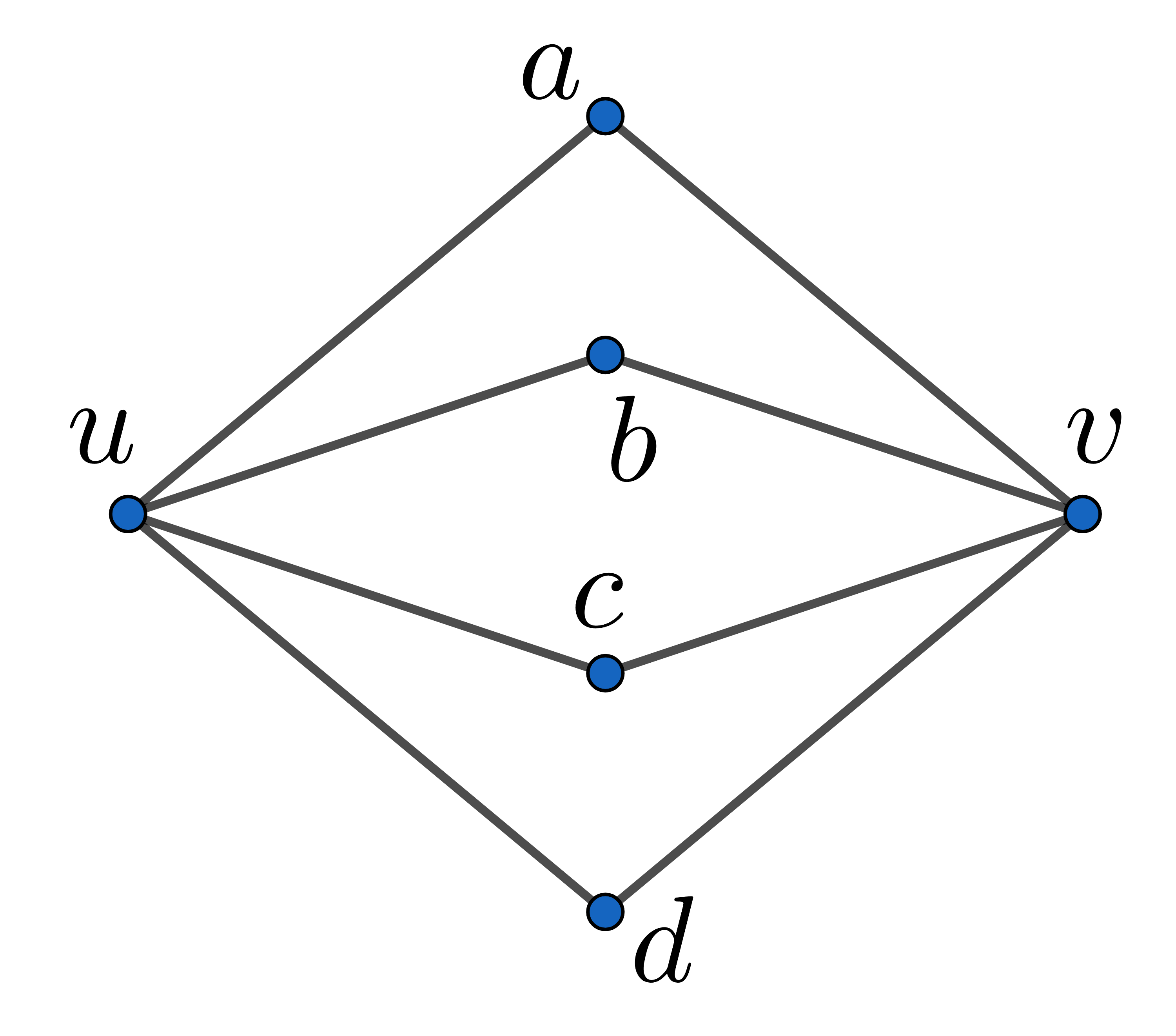}
	\caption{$N(u,v)=\{a,b,c,d\}$, in this cyclic order around $u$.}
	\label{fig:pf01}
\end{figure}

The only neighbours of $u,v$ are $a_1,a_2,\dots,a_r$, hence there cannot be any vertices inside of any of the cycles \begin{equation}
	\label{eq:4cy}
	u,a_i,v,a_{i+1}, \qquad 1\leq i\leq r
\end{equation}
(taking $a_{r+1}=a_1$), otherwise we would have an instance of a cycle $\calc$ satisfying $|S(\calc)|\in\{1,2\}$ in the notation of Lemma \ref{le:sc}. Therefore, $|V(G)|=2+r$ so that the only possibility is $r=4$ and $G$ is the octahedron.
\end{proof}

\begin{comment}
On the other hand, Figure \ref{fig:0123} shows an illustration of a $3$-regular polyhedron with $A=\{0,1,2\}$, and $4$- and $5$-regular polyhedra with $A=\{0,1,2,3\}$.
\begin{figure}[ht]
	\centering
	\begin{subfigure}{0.25\textwidth}
		\centering
		\includegraphics[width=2.cm]{0123a.png}
		\caption{$r=3$, $A=\{0,1,2\}$.}
		\label{fig:0123a}
	\end{subfigure}
	\hfill
	\begin{subfigure}{0.25\textwidth}
		\centering
		\includegraphics[width=2.5cm]{0123b.png}
		\caption{$r=4$, $A=\{0,1,2,3\}$.}
		\label{fig:0123b}
	\end{subfigure}
	\hfill
	\begin{subfigure}{0.48\textwidth}
		\centering
		\includegraphics[width=4.5cm]{0123.png}
		\caption{$r=5$, $A=\{0,1,2,3\}$.}
		\label{fig:0123c}
	\end{subfigure}
	\caption{In Figure \ref{fig:0123c}, the thick edges highlight the copy of $K(2,3)$.}
	\label{fig:0123}
\end{figure}
\end{comment}

We continue by stating two basic observations for quick reference.
\begin{lemma}
	\label{le:0A}
Let $G$ be an $r$-regular graph of order $p$. If $p>r^2+1$, then $0\in A$.
\end{lemma}
\begin{proof}
If $G$ is disconnected, then clearly $0\in A$. Now let $G$ be connected and fix $v\in V(G)$. There are $r$ vertices at distance $1$ from $v$, and at most $r(r-1)$ at distance $2$ from $v$. Hence as soon as
\[p>1+r+r(r-1)=r^2+1,\]
there exists $u\in V(G)$ at distance $3$ from $v$, so that $N(u,v)=\emptyset$, thus $0\in A$.
\end{proof}

\begin{lemma}[{\cite[Corollary 3.3]{maffucci2025classification}}]
	\label{le:2A}
Let $G$ be a planar graph. Then we have $2\in A$ if and only if $G$ contains a $4$-cycle.
\end{lemma}

The following result is a consequence of \cite[Theorem 1.3]{maffucci2025classification}. An alternative self-contained proof may be found in Appendix \ref{app:c}.
\begin{prop}
	\label{prop:1A}
Let $G$ be an $r$-regular polyhedron other than the tetrahedron, cube, octahedron, and icosahedron. Then $1\in A$.
\end{prop}

The final ingredients needed for the proof of Theorem \ref{thm:35r} concern $5$-regular polyhedra.

\begin{prop}
	\label{prop:30}
	Let $G$ be a plane, connected, $5$-regular $(p,q)$-graph with no quadrangular regions. Call $q_2$ and $q_0$ the number of edges of $G$ that lie on $2$ and $0$ triangular regions respectively. Then we have
	\begin{equation}
		\label{eq:21}
		q_2\geq 15+\frac{q}{2}+q_0\geq 30,
	\end{equation}
	and these bounds are sharp.
	\\
	It follows that every planar, $5$-regular graph contains a $4$-cycle.
\end{prop}
\begin{proof}
	Let $p,q,f$ be the number of vertices, edges, and regions of $G$. For $i\geq 3$, let $f_i$ be the number of regions of $G$ bounded by a circuit of length $i$. Note that each of these circuits has at least $3$ distinct vertices, thus those of length $3$ and $4$ are always cycles. By the handshaking lemma we have $5p=2q$, that combined with Euler's formula yields
	\begin{equation}
		\label{eq:qf}
		20+6q=10f.
	\end{equation}
	By the handshaking lemma applied to the dual of $G$ (which may be a multigraph) we write
	\[20+3\sum_{i\geq 3}if_i=20+6q=10f=10\sum_{i\geq 3}f_i,\]
	whence
	\[f_3=20+\sum_{i\geq 4}(3i-10)f_i.\]
	By assumption, $f_4=0$ thus
	\[f_3=20+\sum_{i\geq 5}(3i-10)f_i\geq 20+5\sum_{i\geq 5}f_i=20+5(f-f_3).\]
	It follows that
	\begin{equation}
		\label{eq:56}
		f_3\geq\frac{5f}{6}+\frac{20}{6}.
	\end{equation}
	Call $q_2,q_1,q_0$ the number of edges of $G$ that lie on $2,1,0$ triangular regions respectively. If we count all edges of all triangular regions in $G$, then on one hand the sum is $3f_3$, and on the other hand the sum is $2q_2+q_1$. We may thus write
	\[\begin{cases}
		3f_3=2q_2+q_1
		\\
		q=q_2+q_1+q_0.
	\end{cases}\]
	Subtracting the second equation from the first and applying \eqref{eq:56} and \eqref{eq:qf},
	\[q_2-q_0=3f_3-q\geq \frac{5f}{2}+10-q=\frac{10+3q}{2}+10-q=15+\frac{q}{2}.\]
	We have obtained the first inequality in \eqref{eq:21}. The second one is derived since $q\geq 30$ holds for any $5$-regular, planar graph. The bounds are sharp since the icosahedron verifies $q=q_2=30$.
	
	To prove the last statement of this proposition, let $G$ be a connected component of a plane, $5$-regular graph, containing no quadrangular regions. As we have seen $G$ contains (at least $30$) edges lying on two triangular regions, thus $G$ has pairs of adjacent triangular regions, hence it contains a $4$-cycle.
\end{proof}

We are now in a position to prove Theorem \ref{thm:35r}.
\begin{proof}[Proof of Theorem \ref{thm:35r}]
Let $G$ be a cubic polyhedron. Combining Lemmas \ref{le:0123}, \ref{le:0A}, and \ref{le:2A}, and Proposition \ref{prop:1A}, and inspecting the few cases with up to $10$ vertices, we conclude that apart from the tetrahedron and the cube, we have either $A=\{0,1,2\}$ or $A=\{0,1\}$, depending on whether $G$ contains a $4$-cycle or not. We note that if a cubic polyhedron contains a $4$-cycle, then all other vertices lie on one side of it, due to Lemma \ref{le:sc}. Moreover, a cubic polyhedron other than the tetrahedron cannot have a pair of adjacent triangular faces, due to $3$-connectivity. Hence in a cubic polyhedron other than the tetrahedron, all $4$-cycles are facial. It follows that if $G$ is a cubic polyhedron other than the tetrahedron and cube, then we have either $A=\{0,1,2\}$ or $A=\{0,1\}$, depending on whether $G$ contains a quadrangular face or not.

Let $G$ be a $4$-regular polyhedron on more than $17$ vertices. We combine Lemmas \ref{le:0123}, \ref{le:0A}, and \ref{le:2A}, and Proposition \ref{prop:1A}, to deduce that $G$ is of type $\{0,1\}$ if and only if it does not contain a $4$-cycle, else it is of type $\{0,1,2\}$ if and only if it does not contain a subgraph isomorphic to $K(2,3)$, else it is of type $\{0,1,2,3\}$. In \cite[Theorem 3]{brin05}, there is a recursive construction of all $3$-connected quadrangulations of the sphere, and hence of their duals, the quartic polyhedra. They are obtained from the pseudo-double-wheels (duals of the antiprisms) by applying two transformations iteratively. We implemented this construction (our code is available on request) to generate all $4$-regular polyhedra of order up to $17$, and check that they also satisfy Theorem \ref{thm:35r}.

Similarly, let $G$ be a $5$-regular polyhedron on more than $26$ vertices. We combine Lemmas \ref{le:0123}, \ref{le:0A}, and \ref{le:2A}, and Propositions \ref{prop:1A} and \ref{prop:30} to prove that $G$ verifies
\[\{0,1,2\}\subseteq A(G)\subseteq\{0,1,2,3,4\}.\]
It follows that
\[A(G)\in\{\{0,1,2\}, \{0,1,2,3\}, \{0,1,2,4\}, \{0,1,2,3,4\}\},\]
as desired. To settle the few remaining cases of $5$-regular polyhedra with up to $26$ vertices, we could proceed as in the case $r=4$, constructing all such polyhedra via the ideas in \cite{hasheminezhad2011recursive}. Here we are going to present an alternative approach, where one proves the following result directly.
\begin{lemma}
	\label{le:tec}
	If $G$ is a $5$-regular polyhedron, then $0\in A$.
\end{lemma}
The proof of Lemma \ref{le:tec} is rather technical, and we have relegated it to Appendix \ref{app:b}. Assuming it, the proof of Theorem \ref{thm:35r} is complete.
\end{proof}

\section{Inspection of $4$-regular polyhedral Deza graphs}
\label{sec:4}
\subsection{Proof of Theorem \ref{thm:4r}}
\label{sec:4r}
According to \cite[Theorem 3]{brin05}, the class of $3$-connected quadrangulations of the sphere may be constructed starting from the pseudo-double-wheels (duals of the antiprisms), and applying the transformations $\mathcal{A}$ and $\mathcal{B}$ depicted in Figure \ref{fig:pf12}. Further, by \cite[Theorem 4]{brin05}, the subclass of $3$-connected quadrangulations with no separating $4$-cycles may be constructed starting from the pseudo-double-wheels and applying transformation $\mathcal{A}$.
\begin{figure}[ht]
	\begin{subfigure}{0.45\textwidth}
		\centering
		\includegraphics[width=5.cm]{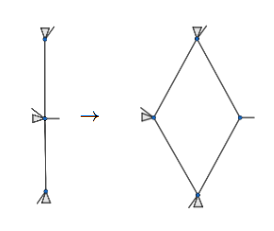}
		\caption{Transformation $\mathcal{A}$.}
		\label{fig:pf12a}
	\end{subfigure}
	\hfill
	\begin{subfigure}{0.45\textwidth}
		\centering
		\includegraphics[width=7.cm]{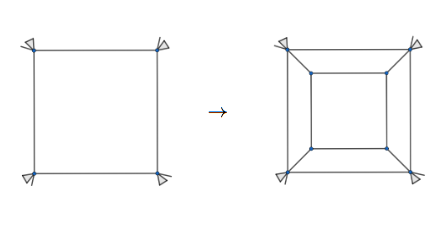}
		\caption{Transformation $\mathcal{B}$.}
		\label{fig:pf12b}
	\end{subfigure}
	\caption{Transformations to construct all $3$-connected quadrangulations. A half-edge indicates an edge that is necessarily there, while a triangle on a vertex indicates one or more edges that might be there.}
	\label{fig:pf12}
\end{figure}

Let $G$ be a $4$-regular polyhedral Deza graph, different from those of Figure \ref{fig:s4}. The dual $G^*$ is a $3$-connected quadrangulation of the sphere. We claim that $G^*$ has no separating $4$-cycles. By contradiction, if $G^*$ has a separating $4$-cycle, then transformation $\mathcal{B}$ of Figure \ref{fig:pf12b} was applied at least once in the construction of $G^*$. After the last application of $\mathcal{B}$, by construction there exists in $G^*$ a face containing only vertices of degree $3$. We note that any subsequent application of $\mathcal{A}$ does not increase the degree of existing vertices. Hence in $G^*$ there exists a face containing only vertices of degree $3$. Therefore, in $G$ there exists a pair of adjacent triangular faces, thus $G$ contains a $4$-cycle, contradicting Proposition \ref{prop:1}.

It follows that $G^*$ is indeed a $3$-connected quadrangulation of the sphere with no separating $4$-cycles. A graph $\Gamma$ is the radial graph of a polyhedron if and only if $\Gamma$ is a $3$-connected quadrangulation of the sphere with no separating $4$-cycles \cite{mafpo3}. Thus $G^*$ is the radial graph of a polyhedron $H$. Equivalently,
\[G=\cm(H).\]

Now by Proposition \ref{prop:1}, $G$ contains no $4$-cycles. Hence $G$ has no quadrangular faces, thus $G^*$ has no vertices of degree $4$, so that by definition of radial graph, $H$ has no vertices of degree $4$ and no quadrangular faces. Moreover, since $G$ has no pairs of adjacent triangular faces, then $G^*$ has no pairs of adjacent vertices of degree $3$, hence in $H$ there is no vertex of degree $3$ lying on a triangular face.

Vice versa, let $H$ be a polyhedron with no vertices of degree $4$, no quadrangular faces, and no vertices of degree $3$ lying on a triangular face. Write $G:=\cm(H)$. Then $G^*=\calr(H)$ is a $3$-connected quadrangulation of the sphere with no separating $4$-cycles, no vertices of degree $4$, and no pair of adjacent vertices of degree $3$. Equivalently, $G$ is a quartic polyhedron with no quadrangular face and no pair of adjacent triangular faces.

Lastly, assume by contradiction that $G$ has a $4$-cycle. Since $G$ is $3$-connected and $4$-regular, by Lemma \ref{le:sc} we have one of the two scenarios in Figure \ref{fig:pf18}. Curved lines represent paths in $G$ (not necessarily edges). The labels are faces of $G$, corresponding to vertices of $G^*$. Lowercase letters correspond to vertices of $H$, and uppercase letters to faces of $H$. In Figure \ref{fig:pf18a}, since $G^*$ is bipartite, its vertices $U,b,Y,c$ are distinct. Then $U,b,Y,c$ is a separating $4$-cycle in $G^*$, contradiction. In Figure \ref{fig:pf18b}, we note that $G^*$ contains odd cycles, hence it is not bipartite, contradiction.

Hence $G$ is a quartic polyhedron with no $4$-cycles, thus it is a Deza graph by Proposition \ref{prop:1}. The proof of Theorem \ref{thm:4r} is complete.
\begin{figure}[ht]
	\begin{subfigure}{0.45\textwidth}
		\centering
		\includegraphics[width=3.cm]{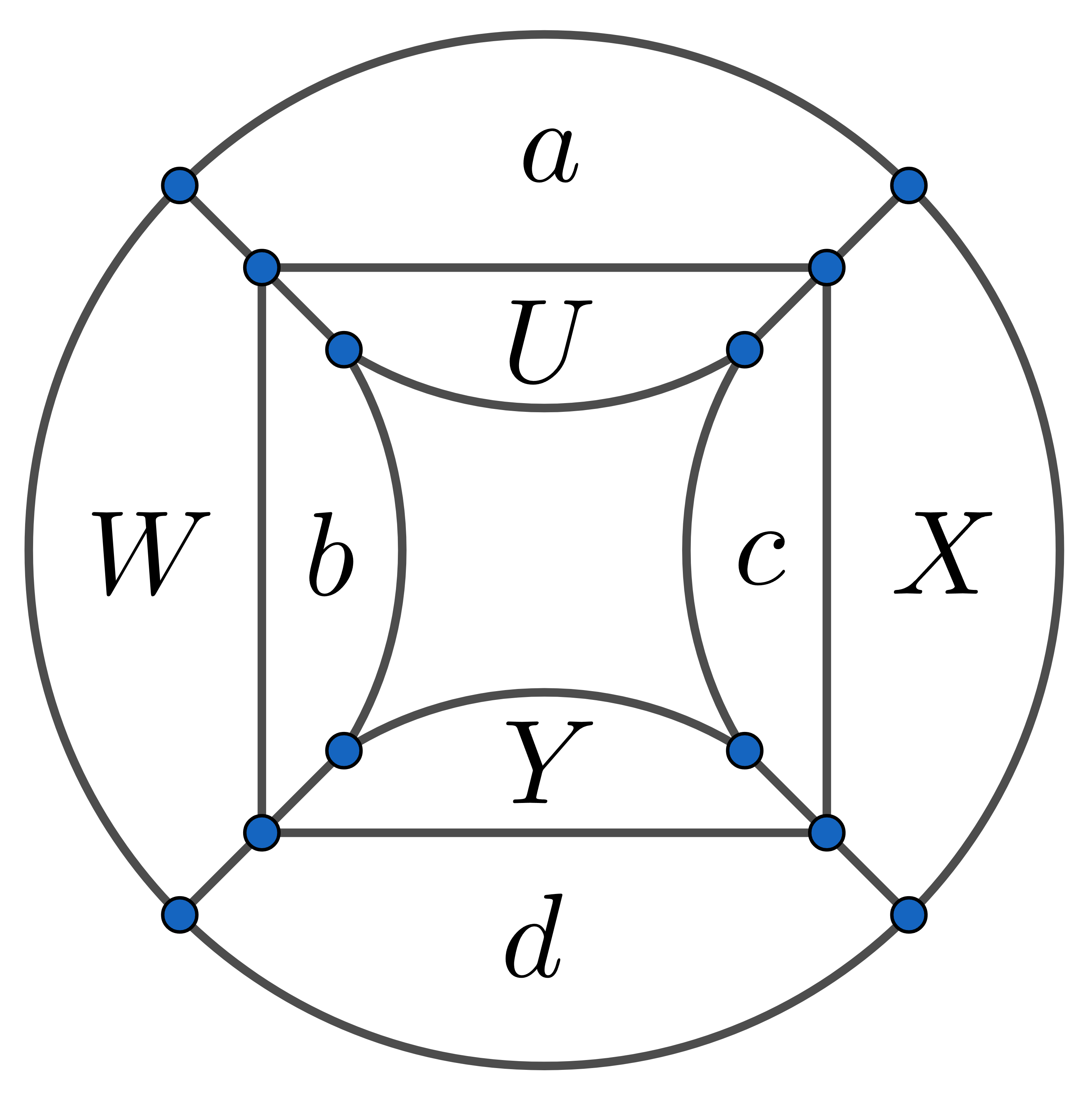}
		\caption{}
		\label{fig:pf18a}
	\end{subfigure}
	\hfill
	\begin{subfigure}{0.45\textwidth}
		\centering
		\includegraphics[width=3.cm]{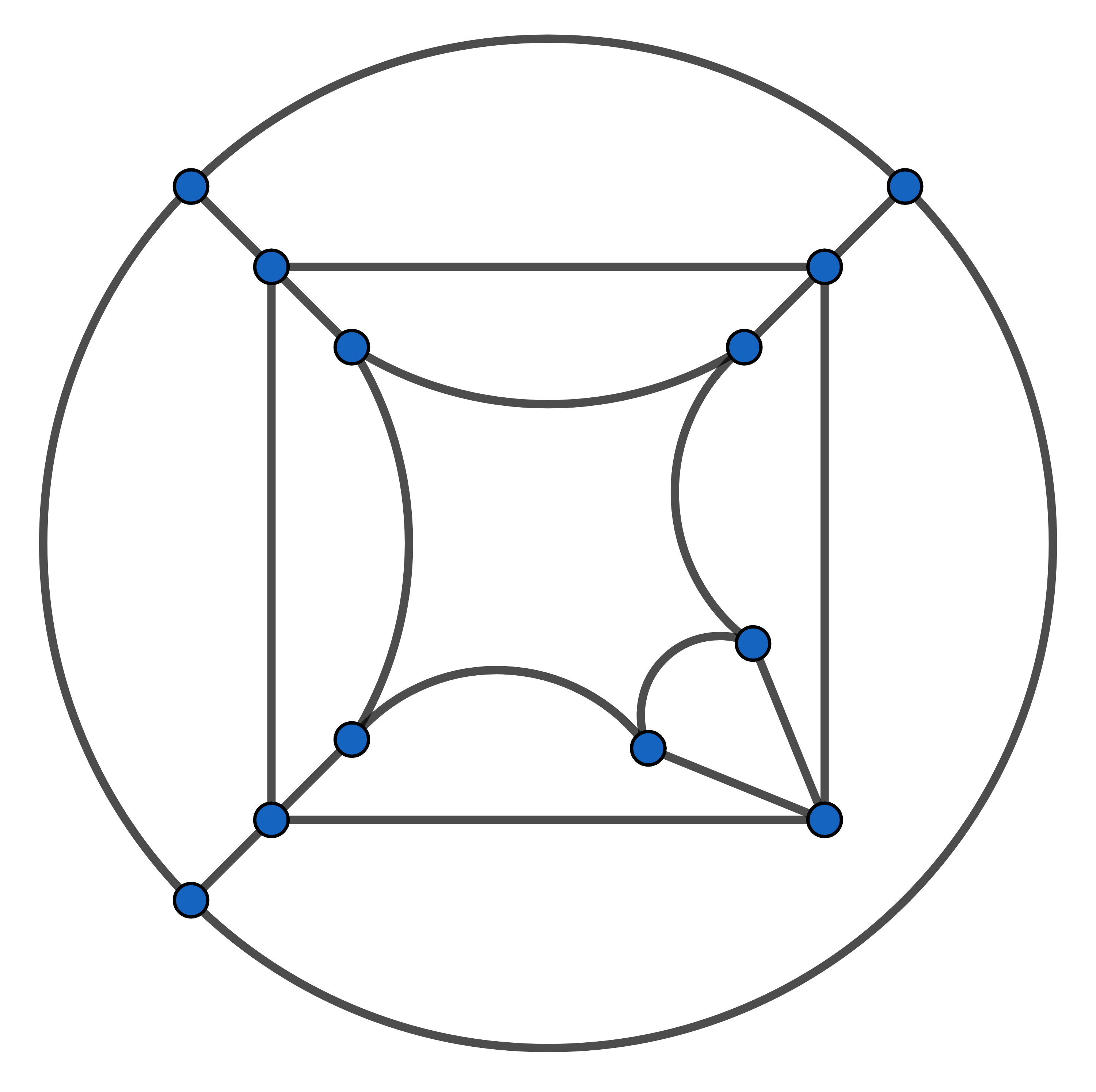}
		\caption{}
		\label{fig:pf18b}
	\end{subfigure}
	\caption{Assuming that $G$ contains a separating $4$-cycle leads to a contradiction.}
	\label{fig:pf18}
\end{figure}

\subsection{Proof of Theorem \ref{thm:4max}}
\label{sec:4max}
We need an auxiliary result, of similar kind to Proposition \ref{prop:30}.
\begin{prop}
	\label{prop:1223}
	Let $G$ be a $4$-regular, planar, Deza $(p,q)$-graph on $f$ faces, different from those of Figure \ref{fig:s4}. Denote by $f_3$ the number of triangular faces in $G$, and by $q_1$ the number of edges lying on a triangular face. Then we have
	\begin{equation}
		\label{eq:f3}
		\frac{f}{2}+4\leq f_3\leq\frac{2f}{3}-\frac{4}{3}
	\end{equation}
	and
	\begin{equation}
		\label{eq:q1}
		q_1\geq\frac{3q}{4}+15,
	\end{equation}
and these bounds are best possible. The lower bound on $f_3$ holds for every $4$-regular polyhedron.
\end{prop}
\begin{proof}
	Let $G$ be a $4$-regular polyhedron. We have $4p=2q$, so that combining with Euler's formula and the handshaking lemma on the dual of $G$ we obtain
	\[8+\sum_{i\geq 3}if_i=8+2q=4f=4\sum_{i\geq 3}f_i.\]
	Rearranging, $f_3=8+\sum_{i\geq 5}(i-4)f_i$. Hence any $4$-regular polyhedron satisfies
	\begin{equation}
	\label{eq:sum}
	f_3=8+\sum_{i\geq 5}(i-4)f_i\geq 8+\sum_{i\geq 5}f_i=8+(f-f_3),
	\end{equation}
	thus $f_3\geq f/2+4$.
	
	Call $q_2,q_1,q_0$ the number of edges of $G$ that lie on $2,1,0$ triangular faces respectively. Note that $2q_2+q_1=3f_3$. We now impose that $G$ is a Deza graph. By Proposition \ref{prop:1}, we have $q_2=0$, hence $q_1=3f_3$. We deduce that
	\[q_1\geq\frac{3f}{2}+12=\frac{3q}{4}+15.\]
	On the other hand
	\[3f_3=q_1\leq q=2f-4,\]
	so that $f_3\leq 2f/3-4/3$, as claimed.
	
	To see that the bounds of the present proposition are best possible, we note that the line graph of the dodecahedron (i.e., the medial graph of the dodecahedron and icosahedron), sketched in Figure \ref{fig:vficos}, satisfies $f=32$, $f_3=20$, and $q=q_1=60$, hence the inequalities in \eqref{eq:f3} and \eqref{eq:q1} are equalities in this example.
	\begin{figure}[ht]
	\centering
	\includegraphics[width=4.25cm]{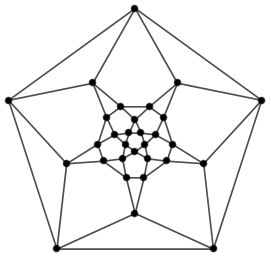}
	\caption{The icosidodecahedron, line graph of the dodecahedron is a $4$-regular, planar, Deza $(30,60)$-graph on $32$ faces, of which $20$ triangular and $12$ pentagonal.}
	\label{fig:vficos}
	\end{figure}
	
\end{proof}

Combining the two inequalities in \eqref{eq:f3}, we obtain $f\geq 32$. Hence a $4$-regular, planar Deza graph other than those in Figure \ref{fig:s4} satisfies $f\geq 32$, thus $q\geq 60$ and $p\geq 30$, as in the penultimate row of Table \ref{t:1}. If these inequalities are equalities, then $q=q_1=60$, and inspecting the proof of Proposition \ref{prop:1223}, we get
\[\sum_{i\geq 5}(i-4)f_i=\sum_{i\geq 5}f_i,\]
i.e.
\[\sum_{i\geq 6}(i-5)f_i=0, \quad\text{ thus }\quad f_i=0 \ \ \forall i\geq 6.\]
Hence in $G$ each face is either triangular or pentagonal, and each triangular face is adjacent only to pentagonal faces, and vice versa. Thus it is easy to see when constructing the graph in Figure \ref{fig:vficos}, that this is in fact the unique smallest quartic planar Deza graph, other than those of Figure \ref{fig:s4}.
% The inequality in is sharp in this case, i.e. $f-f_3=f_5=12$, and $q_1=q$.% It contains no $4$-cycles. It is the line graph of the icosahedron, consistent with Theorem .

We are ready to prove Theorem \ref{thm:4max}.
\begin{proof}[Proof of Theorem \ref{thm:4max}]
Let $G$ be a $4$-regular, planar, Deza $(p(G),q(G))$-graph with $f(G)$ faces (distinct from those in Figure \ref{fig:s4}), with number of triangular faces $f_3(G)$ maximised w.r.t. to $f(G)$. By Proposition \ref{prop:1223}, we have
\[f_3(G)=\frac{2f(G)}{3}-\frac{4}{3} \qquad \text{ and }\qquad q_1(G)=q(G),\]
where $q_1(G)$ counts the edges of $G$ lying on a triangular face. Since $q_1(G)=q(G)$, in $G$ every edge lies on one triangular and one non-triangular face.

%Let $G^*$ be the dual polyhedron of $G$. We claim that $G^*$ is the radial graph of a polyhedron. A graph $\Gamma$ is the radial graph of a polyhedron if and only if $\Gamma$ is a $3$-connected quadrangulation of the sphere with no separating $4$-cycles. Since $G$ is a $4$-regular polyhedron, then $G^*$ is a $3$-connected quadrangulation. It thus suffices to show that $G^*$ has no separating $4$-cycles. Note that this is equivalent to the following condition: in $G$ there do not exist four faces $F_1,F_2,F_3,F_4$ such that for each $i=1,2,3,4$, $F_{i}$ is adjacent to $F_{(i\mod 4+1)}$, and there is no common vertex to $F_1,F_2,F_3,F_4$. By contradiction, let the condition be true for $G$. Since we have already shown that in $G$ every edge lies on one triangular and one non-triangular face, then we may assume that $F_2,F_4$ are triangular faces, and $F_1,F_3$ non-triangular. Now there is a common edge between $F_1,F_2$, and there is a common edge between $F_2,F_3$. Since $F_2$ is a triangle, then there is a common vertex $x$ between $F_1,F_2,F_3$. Likewise, there is a common vertex $y$ between $F_3,F_4,F_1$. Moreover, as there is no common vertex to $F_1,F_2,F_3,F_4$, we have $x\neq y$. Thereby, $x,y$ are common vertices to $F_1,F_3$. Since $G$ is a polyhedron, we deduce that $xy\in E(G)$. However as in $G$ every edge lies on a triangular face, we have reached a contradiction. We have succeeded in proving that there exists a polyhedron $H$ such that
By Theorem \ref{thm:4r}, there exists a polyhedron $H$ with no vertices of degree $4$, no quadrangular faces, and no vertex of degree $3$ lying on a triangular face, such that
\[G=\cm(H).\]
The dual graph $G^*$ of $G$ is the radial graph of $H$.

We will denote by $p,q,f$ the number of vertices, edges, and faces of $H$. Similarly, $p_i$ and $f_i$ stand respectively for the number of vertices of degree $i$ and the number of faces of length $i$ in $H$. Since $f_3(G)=2f(G)/3-4/3$, then $p_3(G^*)=2p(G^*)/3-4/3$, so that by definition of radial graph we have
\[p_3+f_3=\frac{2(p+f)}{3}-\frac{4}{3}.\]
Due to Euler's formula, we obtain
\[3p_3+3f_3=2p+2f-4=2q+4-4=2q.\]

For $i=1,2$, let us call $d_i$ the number of edges of $H$ incident to exactly $i$ vertices of degree $3$, $t_i$ the number of edges of $H$ lying exactly on $i$ triangular faces, and $o$ the number of edges of $H$ that are not adjacent to vertices and degree $3$ and do not lie on a triangular face. Since we have shown that in $H$ there is no vertex of degree $3$ lying on a triangular face, we have $q=d_2+d_1+t_2+t_1+o$.

If we count all incident edges to a vertex of degree $3$ in $H$, we will have counted twice each edge connecting two vertices of degree $3$, and once each edge incident to exactly one vertex of degree $3$, thus $3p_3=2d_2+d_1$. Similarly, $3f_3=2t_2+t_1$. Collecting the information we have on $H$,
\[\begin{cases}
3p_3+3f_3=2q\\
q=d_2+d_1+t_2+t_1+o\\
3p_3=2d_2+d_1\\
3f_3=2t_2+t_1.
\end{cases}\]
Combining these equations yields
\[(2d_2+d_1)+(2t_2+t_1)=2(d_2+d_1+t_2+t_1+o),\]
thus
\[d_1+t_1+2o=0,\]
so that $d_1=t_1=o=0$, hence
\[q=d_2+t_2.\]
Recalling that in $H$ there is no vertex of degree $3$ lying on a triangular face, we deduce that there are two cases. Either $q=d_2$, so that $H$ is a cubic polyhedron, or $q=t_2$, so that $H$ is a triangulation of the sphere. Recalling that $G=\cm(H)$ if and only if $G=\cm(H^*)$, and cubic polyhedra are the duals of the triangulations, we conclude that $G$ is the medial graph of a cubic polyhedron $H$. Since in $H$ no vertex of degree $3$ lies on a triangular face, then $H$ in fact has no triangular faces. We have already shown that $H$ has no quadrangular faces either, thus $H$ is of girth $5$. Since $H$ is a cubic polyhedron, we write
\[G=\cm(H)=\cl(H),\]
as claimed.

Vice versa, let $H$ be a cubic polyhedron of girth $5$, and $G:=\cl(H)=\cm(H)$. Hence $G^*=\calr(H)$ is a $3$-connected quadrangulation with no separating $4$-cycles, no vertices of degree $4$, and such that each edge is incident to one vertex of degree $3$ and one vertex of degree $\geq 5$. Equivalently, $G$ is a quartic polyhedron with no quadrangular faces, and such that each edge lies on one triangular face and one face of length $\geq 5$. It follows that $G$ contains no quadrangular faces and no pairs of adjacent triangular faces. Further, as seen in the proof of Theorem \ref{thm:4r}, $G$ contains no separating $4$-cycles. Therefore, $G$ contains no $4$-cycles at all. By Proposition \ref{prop:1}, $G$ is a Deza graph. According to Proposition \ref{prop:1223}, since $q_1(G)=q(G)$, we have $f_3(G)=2f(G)/3-4/3$, as claimed.
\end{proof}

\subsection{Proof of Proposition \ref{prop:4min}}
\label{sec:4min}
\begin{proof}[Proof of Proposition \ref{prop:4min}]
Let $G$ be a $4$-regular polyhedral Deza $(p(G),q(G))$-graph with $f(G)$ faces (distinct from those in Figure \ref{fig:s4}), with number of triangular faces $f_3(G)$ minimised w.r.t. to $f(G)$. By Proposition \ref{prop:1223}, we have
\[f_3(G)=\frac{f(G)}{2}+4 \qquad \text{ and }\qquad q_0(G)=\frac{q(G)}{4}-15,\]
where for $i=0,1$, $q_i(G)$ is the number of edges in $G$ lying on $i$ triangular faces. Inspecting the proof of Proposition \ref{prop:1223}, we deduce that equality holds in \eqref{eq:sum}, thus
\[\sum_{i\geq 5}(i-4)f_i(G)=\sum_{i\geq 5}f_i(G),\]
i.e.
\[\sum_{i\geq 6}(i-5)f_i(G)=0, \quad\text{ thus }\quad f_i(G)=0 \ \ \forall i\geq 6.\]
We have obtained
\[f(G)=f_3(G)+f_5(G), \qquad f_3(G)=\frac{f(G)}{2}+4, \qquad f_5(G)=\frac{f(G)}{2}-4.\]

By Theorem \ref{thm:4r}, $G=\cm(H)$, where $H$ is a $(p,q)$-polyhedral graph on $f$ faces, with no vertices of degree $4$, no quadrangular faces, and no vertices of degree $3$ lying on a triangular face. Let $p_i,f_i$ be the number of vertices of degree $i$ and faces of length $i$ in $H$ respectively. Since in $G$ every face is either triangular of pentagonal, then in $G^*=\calr(H)$, every vertex has degree either $3$ or $5$, hence in $H$ we have $p=p_3+p_5$ and $f=f_3+f_5$. In particular, in $H$ every vertex of degree $3$ lies only on pentagonal faces.

Furthermore,
\[p_3+f_3=p_3(G^*)=f_3(G)=\frac{f(G)}{2}+4=\frac{p(G^*)}{2}+4=\frac{p+f}{2}+4=\frac{q}{2}+5,\]
where we have used Euler's formula for $H$ in the last equality.

Vice versa, if $H$ is a polyhedron satisfying such properties, then by the equalities above and by Theorem \ref{thm:4r}, $G:=\cm(H)$ is a $4$-regular planar Deza graph verifying $f_3(G)=f(G)/2+4$.
\end{proof}

Note that the graph in Figure \ref{fig:vficos} satisfies both the lower and upper bound for $f_3$ with respect to $f$, hence it is also an illustration of $G$ for Proposition \ref{prop:4min}. In this case $H$ is either the dodecahedron or its dual the icosahedron.

Another example for Proposition \ref{prop:4min} is given in Figure \ref{fig:min}. We have $G=\cm(H)$, and $H$ is a self-dual polyhedral $(26,50)$-graph on $26$ faces. Note that in $H$ we have $p_3=f_3=15$, and accordingly $q/2+5=30$. In $H$ each vertex of degree $3$ lies on three pentagons.
\begin{figure}[ht]
	\begin{subfigure}{0.45\textwidth}
		\centering
		\includegraphics[width=4.5cm]{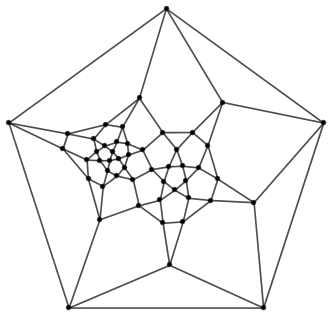}
		\caption{$G$.}
		\label{fig:minG}
	\end{subfigure}
	\hfill
	\begin{subfigure}{0.45\textwidth}
		\centering
		\includegraphics[width=4.5cm]{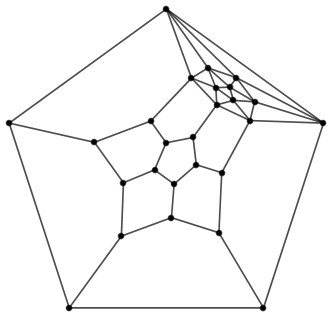}
		\caption{$H$.}
		\label{fig:minH}
	\end{subfigure}
	\caption{An illustration of Proposition \ref{prop:4min}.}
	\label{fig:min}
\end{figure}

\section{$4$-regular polyhedra of type $\{0,1,2,3\}$: proof of Proposition \ref{prop:0123}}
\label{sec:0123}
By assumption $G$ is a $4$-regular polyhedron satisfying $A(G)=\{0,1,2,3\}$. Hence we may find $a,b,c,d,e\in V(G)$ such that
\[N(a,b)=\{c,d,e\}.\]
We consider the $4$-cycles
\[a,c,b,d,\qquad a,d,b,e,\qquad a,e,b,c.\]
At most one of these may be facial, else $G$ would contain a vertex of degree $2$ (this is similar to Lemma \ref{le:0123}).

Let $\calc$ be a non-facial cycle in $G$, and $H$ the subgraph of $G$ generated by all the vertices inside of $\calc$. The vertices of $H$ not adjacent to any vertex on $\calc$ have degree $4$ in $H$, thus in particular the total degree of these vertices is an even number. By the handshaking lemma on $H$, there is hence an even number of edges connecting vertices on $\calc$ with vertices of $H$. By the $3$-connectivity of $G$, this even number is at least $4$.

Suppose for the moment that two of $c,d,e$ are adjacent, w.l.o.g.\ $cd\in E(G)$. At least one of the cycles $a,d,b,e$ and $a,e,b,c$ is not facial, and moreover the vertices $a$, $b$, $c$, and $d$ are each adjacent to one more vertex, and $e$ to two more. By the above considerations and by Lemma \ref{le:sc}, $e$ is adjacent to one of $c,d$, say $d$, so that ultimately the subgraph of $G$ generated by $a,b,c,d,e$ is a square pyramid with apex $d$. In $G$, each of $a$, $b$, $c$, and $e$ has a neighbour inside of the cycle $a,e,b,c$. The reader may refer to Figure \ref{fig:0123pa}.
\begin{figure}[ht]
	\begin{subfigure}{0.45\textwidth}
		\centering
		\includegraphics[width=4.cm]{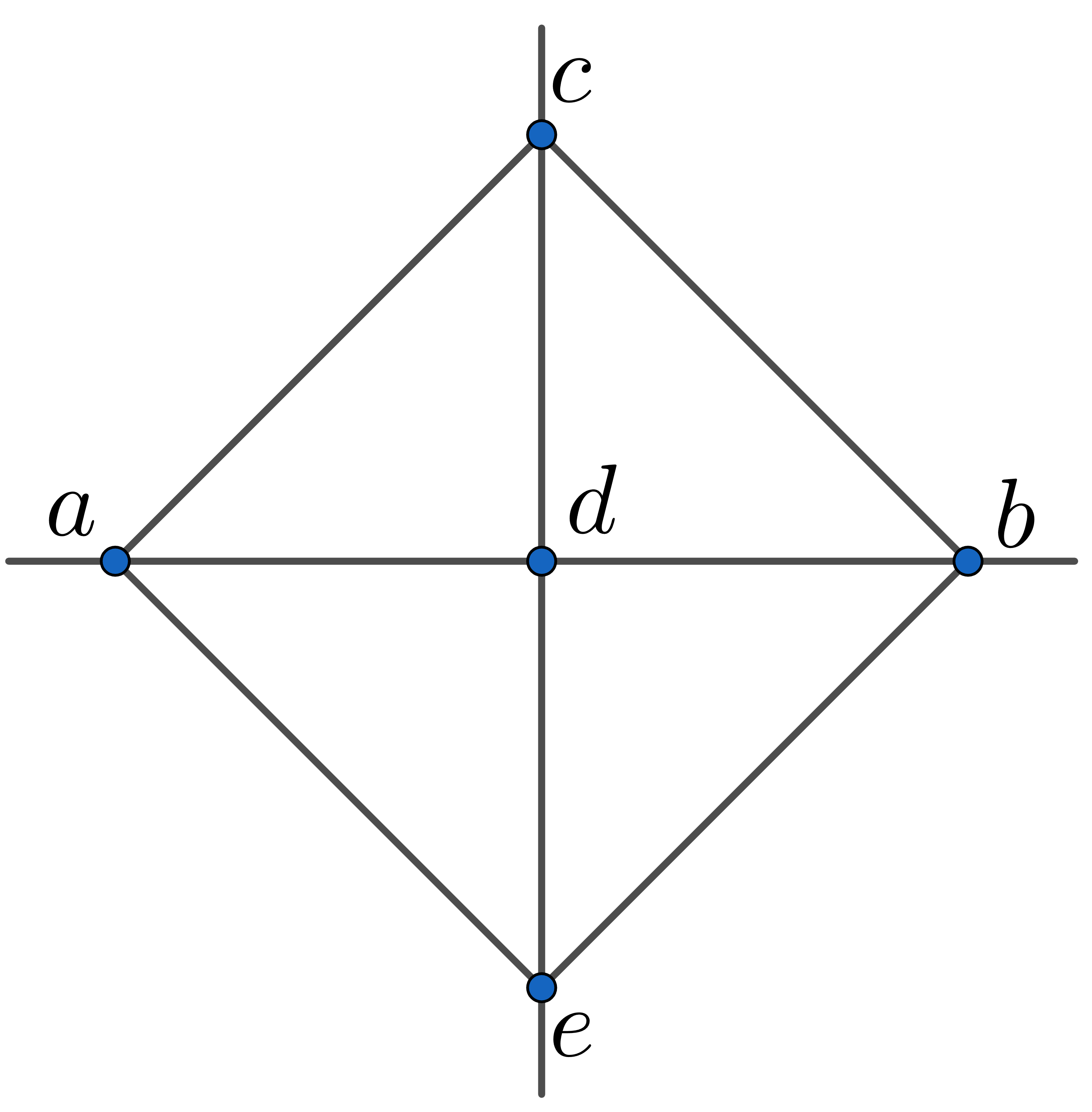}
		\caption{$cd\in E(G)$.}
		\label{fig:0123pa}
	\end{subfigure}
	\hfill
	\begin{subfigure}{0.45\textwidth}
		\centering
		\includegraphics[width=4.cm]{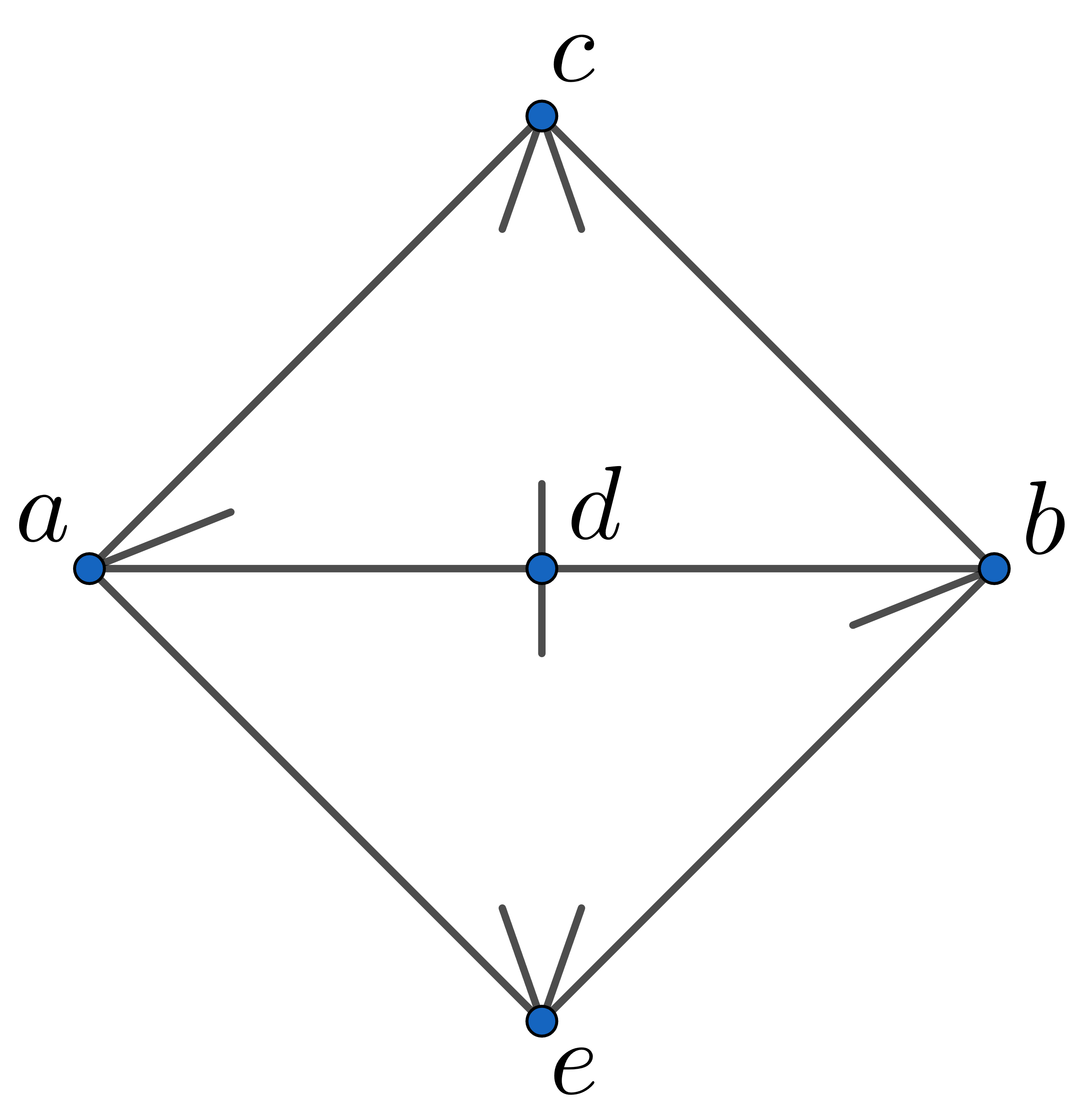}
		\caption{$c,d,e$ are pairwise non-adjacent.}
		\label{fig:0123pb}
	\end{subfigure}
	\caption{Proof of Proposition \ref{prop:0123}.}
	\label{fig:0123}
\end{figure}

The only other possibility is that $c,d,e$ are pairwise non-adjacent. Then by the above considerations and by Lemma \ref{le:sc}, w.l.o.g.\ $[a,e,b,c]$ is a face of $G$, $a,d$ each have one neighbour and $c$ has two neighbours inside of the cycle $a,c,b,d$, and $b,d$ each have one neighbour and $e$ has two neighbours inside of the cycle $a,d,b,e$. The reader may refer to Figure \ref{fig:0123pb}.

Let
\[G_1':=G-e\qquad\text{and}\qquad G_2':=G-c.\]
We have recovered the construction of Definition \ref{def:1} with $c=u(G_1)$, $e=u(G_2)$, $a=x(G_1)=z(G_2)$ $d=y(G_1)=y(G_2)$, $b=z(G_1)=x(G_2)$, $v(G_1)$ the unique neighbour of $a$ inside $a,c,b,d$, $w(G_1)$ the unique neighbour of $d$ inside $a,c,b,d$, $v(G_2)$ the unique neighbour of $b$ inside $a,d,b,e$, and $w(G_2)$ the unique neighbour of $d$ inside $a,d,b,e$. We now obtain $G_1,G_2$ from $G_1',G_2'$ by applying the inverse transformation of \eqref{eq:prime}. Since $G$ is a $4$-regular polyhedron, so are $G_1,G_2$.

It remains to show that $G$ is not a medial graph. By definition of medial graph, if $G=\cm(\Gamma)$, we may recover $\Gamma$ in the following way. We define a vertex for each face of $G$, and edges between vertices corresponding to faces of $G$ that share a vertex but not an edge in $G$. This construction yields the disjoint union of the dual pair of graphs $\Gamma,\Gamma^*$.

Let $G$ contain a square pyramid of apex $d$ and base $a,c,b,e$. Call $F_{bc}$ the face of $G$ that has the common edge $bc$ with $[b,c,d]$, and $F_{ae}$ the face that has the common edge $ae$ with $[a,e,d]$. Since $G$ is not the octahedron, then $\Gamma$ is not the tetrahedron, hence removing from $\Gamma$ the vertices corresponding to $F_{bc},F_{ae}$ yields at least two connected components (one of them containing only the vertices corresponding to the faces $[a,c,d]$ and $[b,e,d]$). Thus $\Gamma$ is not $3$-connected, hence $G$ is not a medial graph of a polyhedron, as claimed.

In the other case above, where $c,d,e$ are pairwise non-adjacent, call $a'$ the neighbour of $a$ distinct from $c,d,e$, $b'$ the neighbour of $b$ distinct from $c,d,e$, $F_{a}$ the face of $G$ containing $a',a,d$, $F_{b}$ the face of $G$ containing $b',b,d$, and $F=[a,e,b,c]$. Removing from $\Gamma$ the vertices corresponding to $F,F_a$ (or $F,F_b$) yields a disconnected graph, thus $G$ is not a medial graph of a polyhedron in this case either. The proof of Proposition \ref{prop:0123} is complete.

\appendix
\section{Revisiting planar Deza graphs: a proof of Proposition \ref{prop:1}}
\label{app:a}
For graphs of connectivity at least $3$, the result follows from Theorem \ref{thm:35r}. Let $G$ be an $r$-regular graph of order $p$, with connectivity at most $2$. The cases $r\leq 2$ are trivial, so henceforth let $r\geq 3$.

We will begin by showing that $0\in A$. If $G$ is disconnected, this is clearly true. Let $G$ be of connectivity $1$, and $v$ be a separating vertex. Since $G$ is $r$-regular for $r\geq 3$ and it has a separating vertex, then there exists $w\in V(G)$ at distance $2$ from $v$. Let $x$ lie in a different connected component to $w$ in $G-v$. Then $N(w,x)=\emptyset$, hence $0\in A$.

Now let $G$ be of connectivity $2$, and $\{u,v\}$ be a cutset. As soon as $p>2+2r$, there exists a vertex $w$ at distance at least $2$ from both $u$ and $v$. Taking $x$ in a different connected component to $w$ in $G-u-v$, we have $N(w,x)=\emptyset$, hence $0\in A$. We inspect by hand the few cases where $p\leq 2+2r$ for $r=3,4,5$.

In any case, $0\in A$, so that we are looking for $\{0,\mu\}$-Deza graphs. We may assume that $G$ is connected by working on a connected component of $G$. Let $G$ be of connectivity $1$, $v$ be a separating vertex, and $w,x$ be neighbours of $v$ lying in distinct connected components of $G-v$. Then $N(w,x)=\{v\}$, so that if $G$ is a Deza graph, then $\mu=1$.

Now let $G$ be of connectivity $2$, $\{u,v\}$ be a cutset, and $w,x$ be neighbours of $u$ lying in distinct connected components of $G-u-v$. Then $\{u\}\subseteq N(w,x)\subseteq\{u,v\}$, so that if $G$ is a Deza graph, then $1\leq\mu\leq 2$.

A connected, $r$-regular, $\{0,2\}$-Deza graph is of connectivity $r$ \cite{brouwer1997vertex}. Hence in a planar, $r$-regular, $\{0,2\}$-Deza graph with $r\geq 3$, every component is a polyhedron. By Theorem \ref{thm:35r}, the only possibilities are a disjoint union of tetrahedra and/or cubes ($r=3$) and a disjoint union of icosahedra ($r=5$).

By Lemma \ref{le:2A}, if $G$ is a planar, $r$-regular graph with $r\geq 3$, then $G$ is a $\{0,1\}$-Deza graph if and only if it does not contain a $4$-cycle (cf.\ \cite[Remark 2]{kabanov2020deza}). On the other hand by Proposition \ref{prop:30}, every planar, $5$-regular graph contains a $4$-cycle. The proof of Proposition \ref{prop:1} is complete.

\section{Proof of Lemma \ref{le:tec}}
\label{app:b}
Let $G$ be a $5$-regular polyhedron of order $p$. There are only three $5$-regular planar graphs with less than $20$ vertices \cite[Table 1]{hasheminezhad2011recursive}, namely the icosahedron and the two polyhedra in Figure \ref{fig:1618}, all of which verify $0\in A$. Thus we will henceforth assume that $p\geq 20$. Let $u,v$ be adjacent vertices of $G$.
\begin{figure}[ht]
	\begin{subfigure}{0.49\textwidth}
		\centering
		\includegraphics[width=3.25cm]{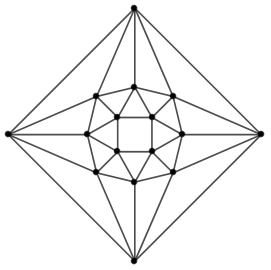}
		\caption{The only $5$-regular planar graph on $16$ vertices.}
		\label{fig:16}
	\end{subfigure}
	\hfill
	\begin{subfigure}{0.49\textwidth}
		\centering
		\includegraphics[width=3.25cm]{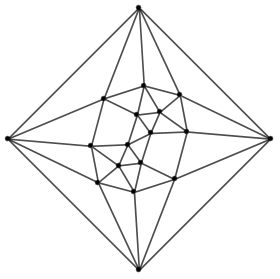}
		\caption{The only $5$-regular planar graph on $18$ vertices.}
		\label{fig:18}
	\end{subfigure}
	\caption{The second and third smallest $5$-regular planar graphs.}
	\label{fig:1618}
\end{figure}
	
We start by considering the case $|N(u,v)|=1$. We write
\[N(u)=\{v,w,u_1,u_2,u_3\} \quad\text{ and }\quad N(v)=\{w,u,v_1,v_2,v_3\}\]
in cyclic order around $u,v$ respectively in the planar immersion of $G$, with
$u_1,u_2,u_3,v_1,v_2,v_3$ distinct.
	
We claim that $u_1u_3,v_1v_3\not\in E(G)$. By contradiction, take $u_1u_3\in E(G)$, as in Figure \ref{fig:pf14a}. Since $0\not\in A$, the vertex $u_2$ has a common neighbour with each of $v_1,v_2,v_3$. Hence by planarity, each of $v_1,v_2,v_3$ is adjacent to one of $u_1,u_3$. We cannot have $u_2$ and all of $v_1,v_2,v_3$ adjacent to $u_1$, else $\deg(u_1)\geq 6$. Therefore, $u_2$ is adjacent to both of $u_1,u_3$, two of $v_1,v_2,v_3$ are adjacent to $u_1$ say, and the third is adjacent to $u_3$. Since $u,u_1$ already have five neighbours each, then by Lemma \ref{le:sc} there are no vertices inside of the cycles $u,u_1,u_2$, $u,u_2,u_3$, and $u_1,u_2,u_3$, thus $\deg(u_2)=3$, contradiction. Thereby, $u_1u_3,v_1v_3\not\in E(G)$ as claimed.
\begin{figure}[ht]
	\begin{subfigure}{0.45\textwidth}
		\centering
		\includegraphics[width=5.75cm]{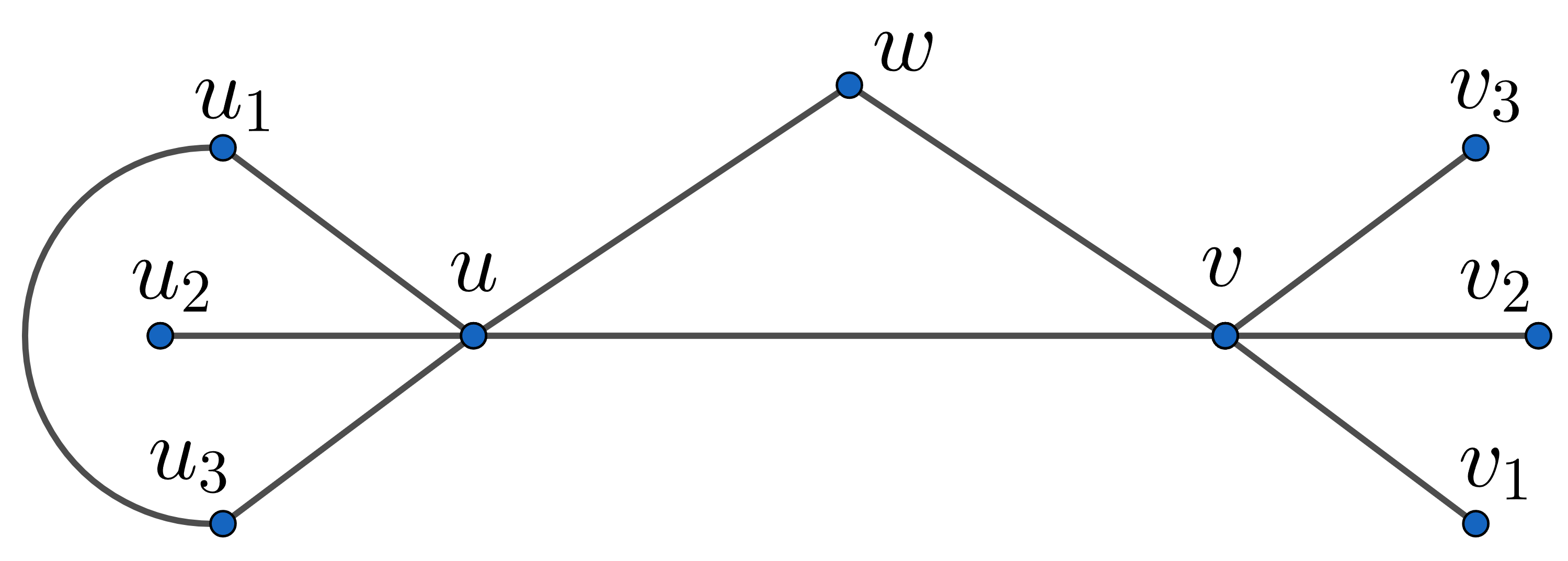}
		\caption{$u_1u_3\in E(G)$ leads to a contradiction.}
		\label{fig:pf14a}
	\end{subfigure}
	\hfill
	\begin{subfigure}{0.45\textwidth}
		\centering
		\includegraphics[width=5.5cm]{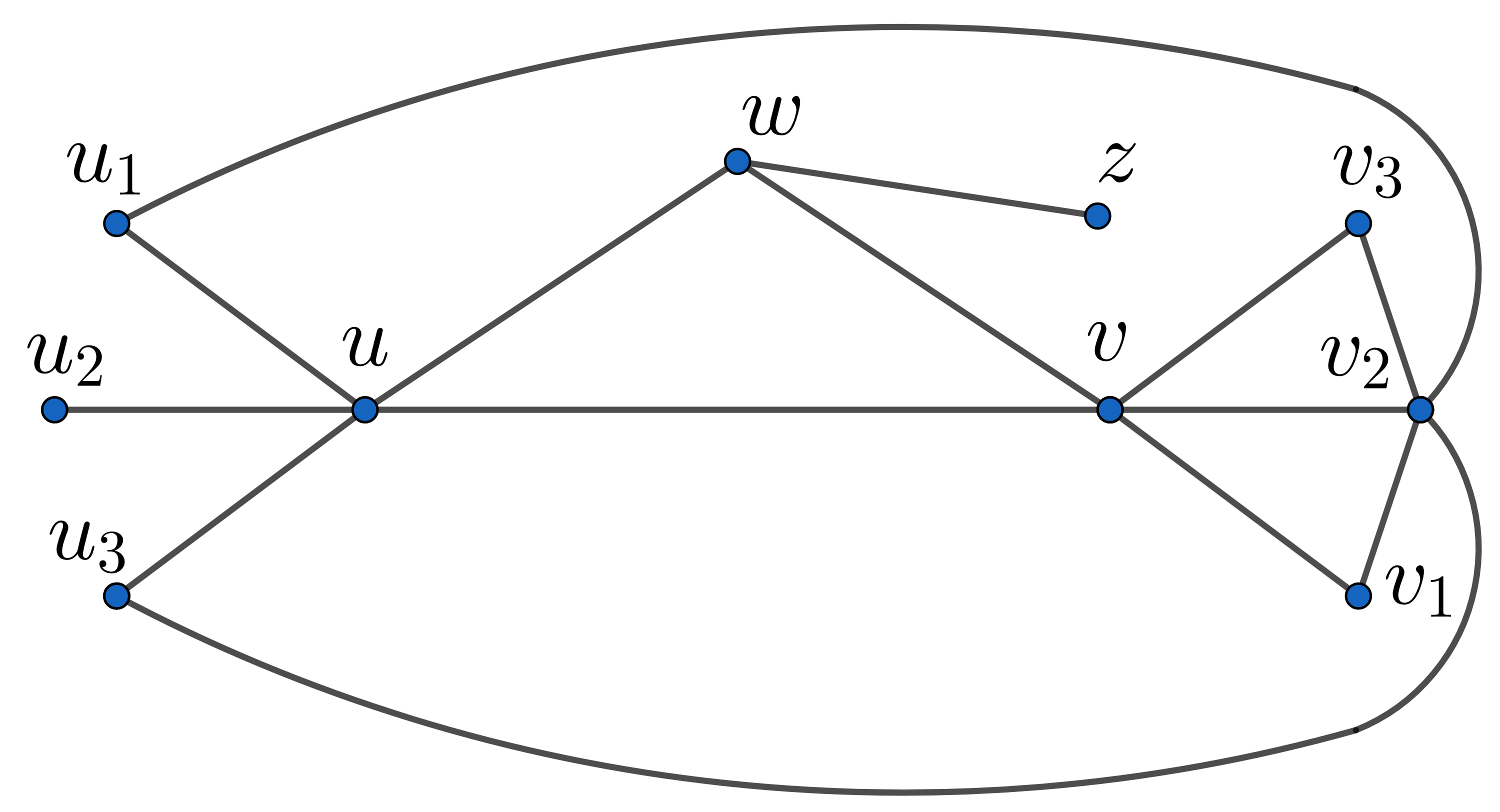}
		\caption{$x=y=v_2$.}
		\label{fig:pf14b}
	\end{subfigure}
	\caption{Case of $|N(u,v)|=1$.}
	\label{fig:pf14}
\end{figure}
	
Let $x\in N(u_1,v_1)$ and $y\in N(v_3,u_3)$. By planarity, and since $u_1u_3,v_1v_3\not\in E(G)$, we have $x=y$. If $x\not\in\{u_2,v_2\}$, then since $0\not\in A$ and by planarity, $x\in N(u_2,v_2)$, thus $\deg(x)\geq 6$, impossible. It follows that $x=y=v_2$, say, as in Figure \ref{fig:pf14b}. By Lemma \ref{le:sc}, the cycle $u,v,w$ is a face, hence there exists $z\in V(G)$ adjacent to $w$ and lying inside the cycle
\[u,w,v,v_3,v_2,u_1.\]
By planarity, and since $u,v,v_2$ already have five neighbours each, we obtain $N(z,v_1)=\emptyset$, contradiction.
	
Now we deal with the case $|N(u,v)|=2$. Let $w,x\in N(u,v)$.

We claim that the cycles $u,v,w$ and $u,v,x$ are facial. By contradiction, suppose that $\calc: u,v,w$ is non-facial. Let $a\neq x$ be a vertex outside of $\calc$. For every vertex $b$ inside of $\calc$, since $N(a,b)\neq\emptyset$, then $N(a,b)$ contains one of $u,v,w$.

If there exist four vertices inside of $\calc$, then two of them are adjacent to $w$, and one each to $u,v$, and $a$ is adjacent to $u,v,w$. There cannot be other vertices inside of $\calc$, or they would not have any common neighbour with $a$. Likewise, there cannot be other vertices outside of $\calc$ apart from $x,a$, or they would not have any common neighbour with any of the vertices inside $\calc$. It follows that $p\leq 9$, impossible.

Hence there are at most three vertices inside of $\calc$. Each of them is of degree $5$, hence by planarity they are all adjacent to one another and to all of $u,v,w$. Hence $G$ has a subgraph isomorphic to $K(3,3)$, contradicting planarity. Thereby, $[u,v,w]$ and $[u,v,x]$ are indeed faces of $G$, as claimed.

Now we write
\[N(u)=\{v,w,u_1,u_2,x\} \quad\text{ and }\quad N(v)=\{u,x,v_1,v_2,w\}\]
in cyclic order around $u,v$ respectively in the planar immersion of $G$, with
$u_1,u_2,v_1,v_2$ distinct. Let $c\in N(u_1,v_1)$. For now we assume that $c\neq w,x$, as in Figure \ref{fig:pf16a}. Apart from the nine
\begin{equation}
	\label{eq:p9}
	u,v,w,x,u_1,u_2,v_1,v_2,c,
\end{equation}
the rest of the vertices of $G$ appear either inside the cycle
\begin{equation}
	\label{eq:6cy}
	u,w,v,v_1,c,u_1,
\end{equation}
or inside the cycle
\[u,x,v,v_1,c,u_1\]
($v_2$ is also inside the first cycle and $u_2$ inside the second). Since $p\geq 20$, then w.l.o.g.\ there are at least six vertices
\[y_1,y_2,y_3,y_4,y_5,y_6\]
inside of \eqref{eq:6cy}. Let $z$ be a vertex distinct from \eqref{eq:p9} outside of \eqref{eq:6cy}. Now for $1\leq i\leq 6$, we have $N(z,y_i)\neq\emptyset$. It follows that $z,y_1,y_2,y_3,y_4,y_5,y_6$ are adjacent to some of $c,u_1,v_1$.
\begin{figure}[ht]
	\begin{subfigure}{0.45\textwidth}
		\centering
		\includegraphics[width=5.25cm]{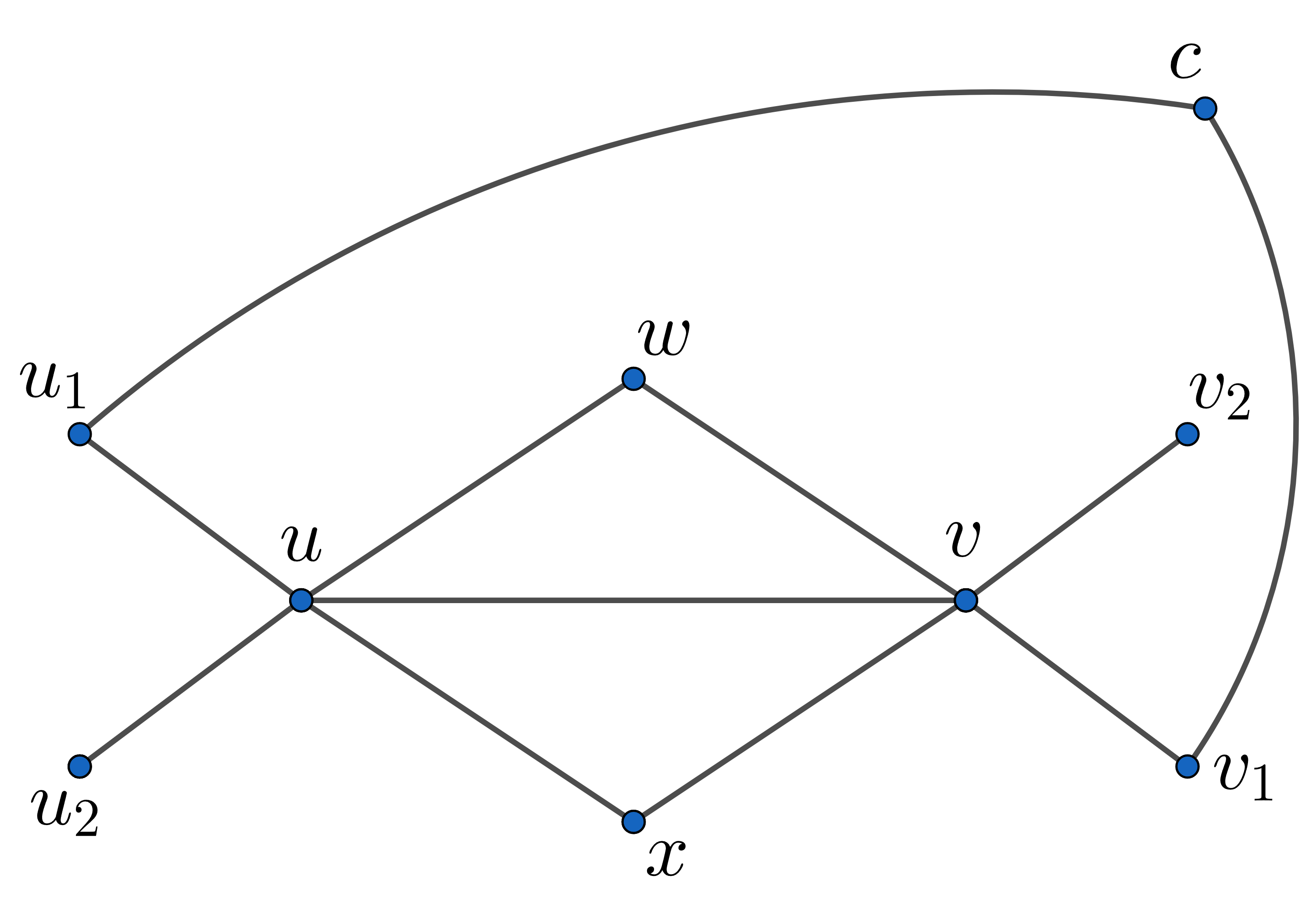}
		\caption{$c\not\in\{w,x\}$.}
		\label{fig:pf16a}
	\end{subfigure}
	\hfill
	\begin{subfigure}{0.45\textwidth}
		\centering
		\includegraphics[width=5.25cm]{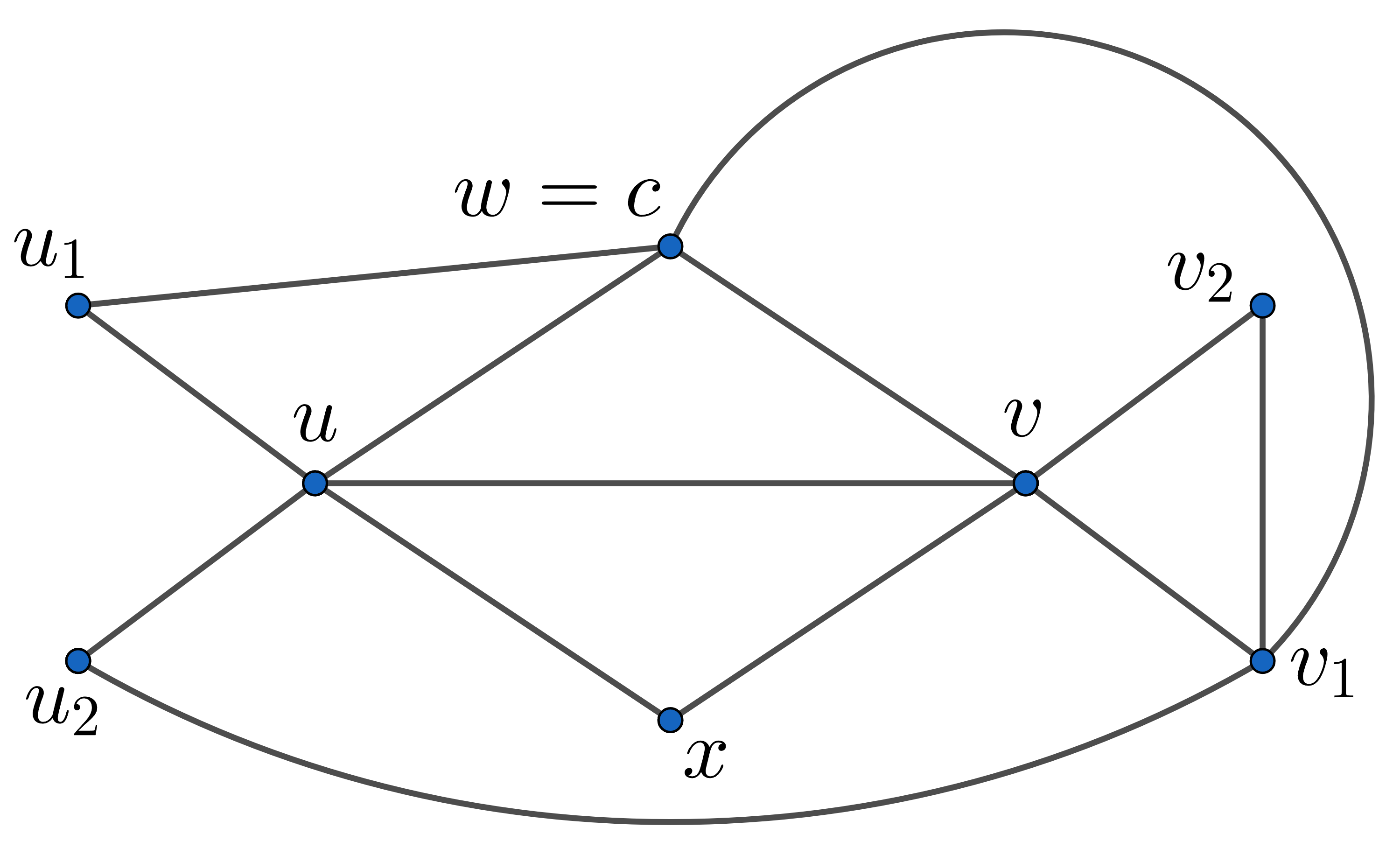}
		\caption{$c=w$.}
		\label{fig:pf16b}
	\end{subfigure}
	\caption{Case of $|N(u,v)|=2$.}
	\label{fig:pf16}
\end{figure}

Each of $c,u_1,v_1$ already has two neighbours. If each of $y_1,y_2,y_3,y_4,y_5,y_6$ is adjacent to at least two of $c,u_1,v_1$, then
\[\deg(c)+\deg(u_1)+\deg(v_1)\geq 2+2+2+6\cdot 2=18,\]
impossible. Hence $y_1$, say, is adjacent to exactly one of $c,u_1,v_1$, that we will denote by $t$. As $0\not\in A$, then every vertex outside of \eqref{eq:6cy}, including $u_2,x$, is adjacent to $t$. Since $\deg(t)=5$, then apart from $u_2,x$, there are no vertices outside of \eqref{eq:6cy}. Since $p\geq 20$, then there are at least eleven vertices inside of \eqref{eq:6cy}, each adjacent to at least one of $c,u_1,v_1$ (or they would have no common neighbours with $x$), leading to
\[\deg(c)+\deg(u_1)+\deg(v_1)\geq 2+2+2+11=17,\]
contradiction.

Now let $c=w$, as in Figure \ref{fig:pf16b}. Since $N(u_2,v_2)\neq\emptyset$, and since $w$ already has four neighbours, we obtain $u_2v_1,v_2v_1\in E(G)$. As $v$ already has five neighbours while $v_2$ only has two, there exists a vertex $y$ inside of the cycle $w,v_1,v_2,v$. Moreover, since every such $y$ has at least one common neighbour with each of $u_1,u_2$, and since $w,v_1$ already have four neighbours each, it follows that there exists a unique vertex $y$ inside of the cycle $w,v_1,v_2,v$. Then $\deg(y)\leq 4$, contradiction.

By Lemma \ref{le:0123}, the last remaining case is $|N(u,v)|=3$. We write $N(u,v)=\{w,x,y\}$. W.l.o.g., the cycles $[u,v,w]$ and $[u,v,y]$ are facial, as in Figure \ref{fig:pf17}. Let $z$ be a vertex inside of the cycle $u,x,v,w$, and $z'$ be a vertex inside of the cycle $u,x,v,y$. Since $N(z,z')\neq\emptyset$, then each of $z,z'$ is adjacent to one of $u,v,x$. Since $u,v$ already have four neighbours and $x$ already two, then there are at most five vertices in $G$ other than $u,v,w,x,y$, thus $p\leq 10$, contradiction.
\begin{figure}[ht]
	\centering
	\includegraphics[width=3.5cm]{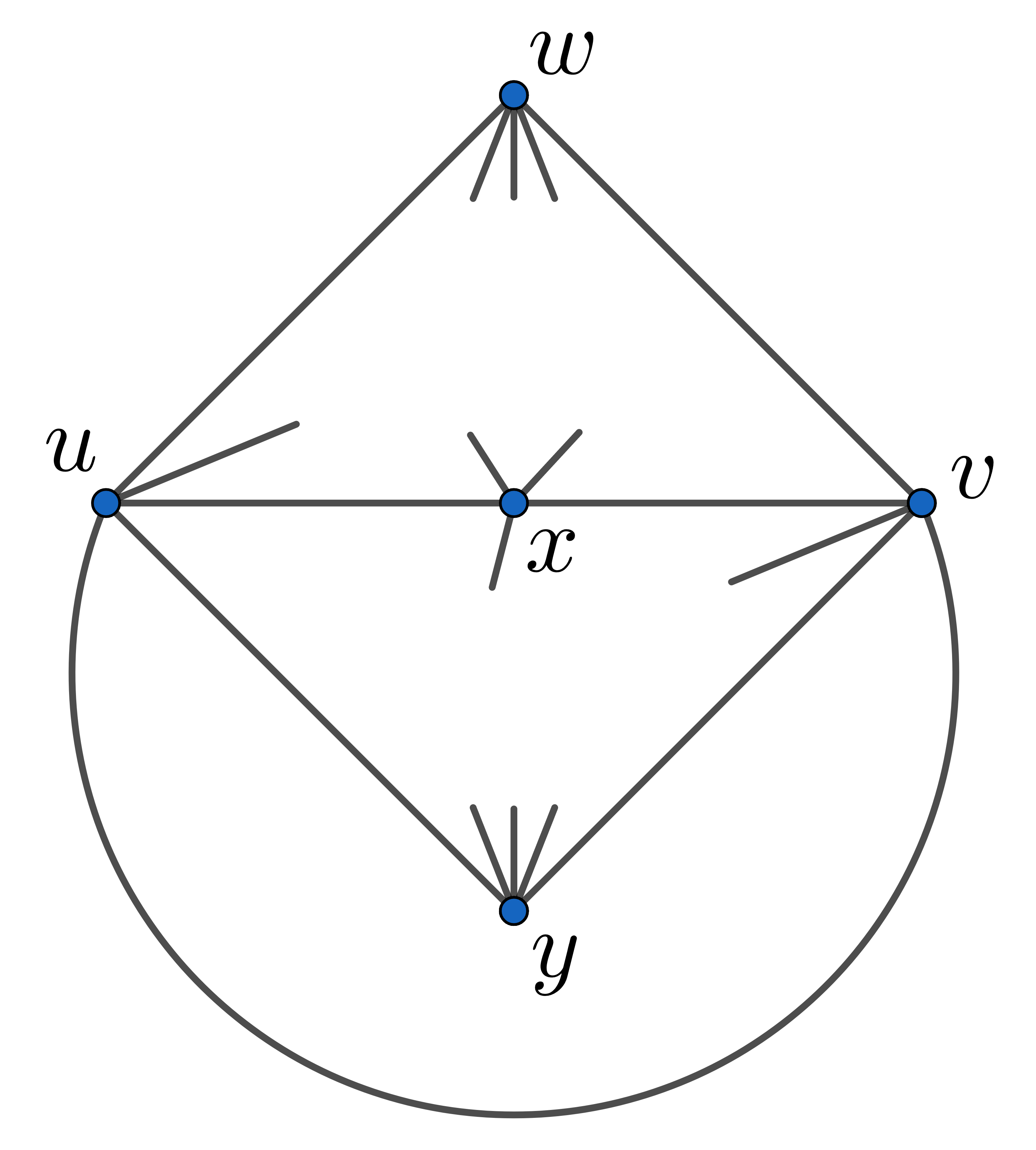}
	\caption{The case $|N(u,v)|=3$.}
	\label{fig:pf17}
\end{figure}

In each case, we have shown that $0\in A$. The proof of Lemma \ref{le:tec} is complete.

\section{Proof of Proposition \ref{prop:1A}}
\label{app:c}
	Let $r=3$ and suppose that $1\not\in A$. Fix $v\in V(G)$ and consider its neighbours $v_1,v_2,v_3$. Since $1\not\in A$, then each pair of $v_1,v_2,v_3$ has a common neighbour other than $v$. Let
	\[N(v_1,v_2)\supseteq\{v,a\}, \quad N(v_1,v_3)\supseteq\{v,b\}, \quad N(v_2,v_3)\supseteq\{v,c\},\]
	as in Figure \ref{fig:pf02}.
	\begin{figure}[ht]
		\centering
		\begin{subfigure}{0.30\textwidth}
			\centering
			\includegraphics[width=2.75cm]{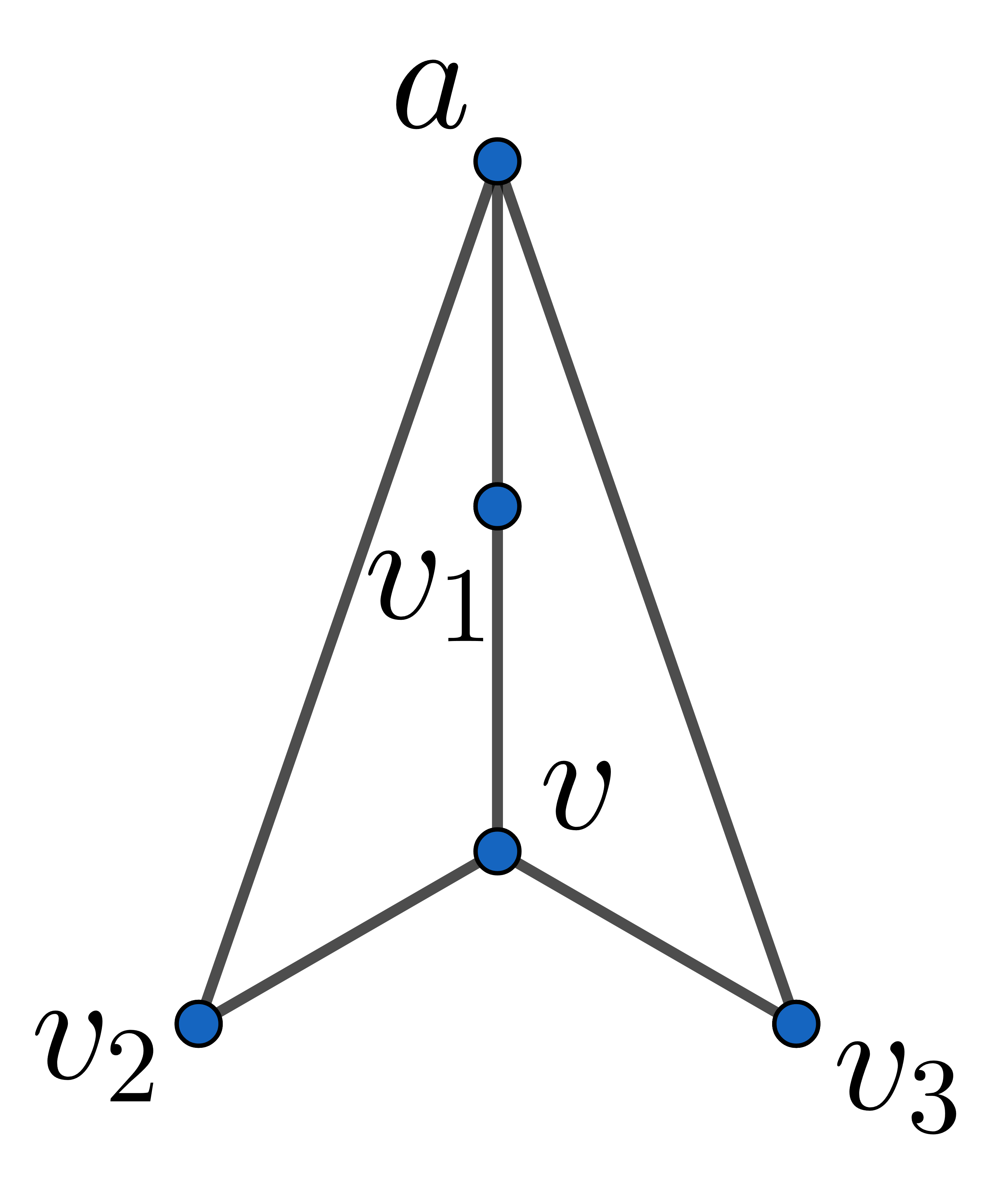}
			\caption{The case $a=b=c$.}
			\label{fig:pf02a}
		\end{subfigure}
		\hfill
		\begin{subfigure}{0.30\textwidth}
			\centering
			\includegraphics[width=2.75cm]{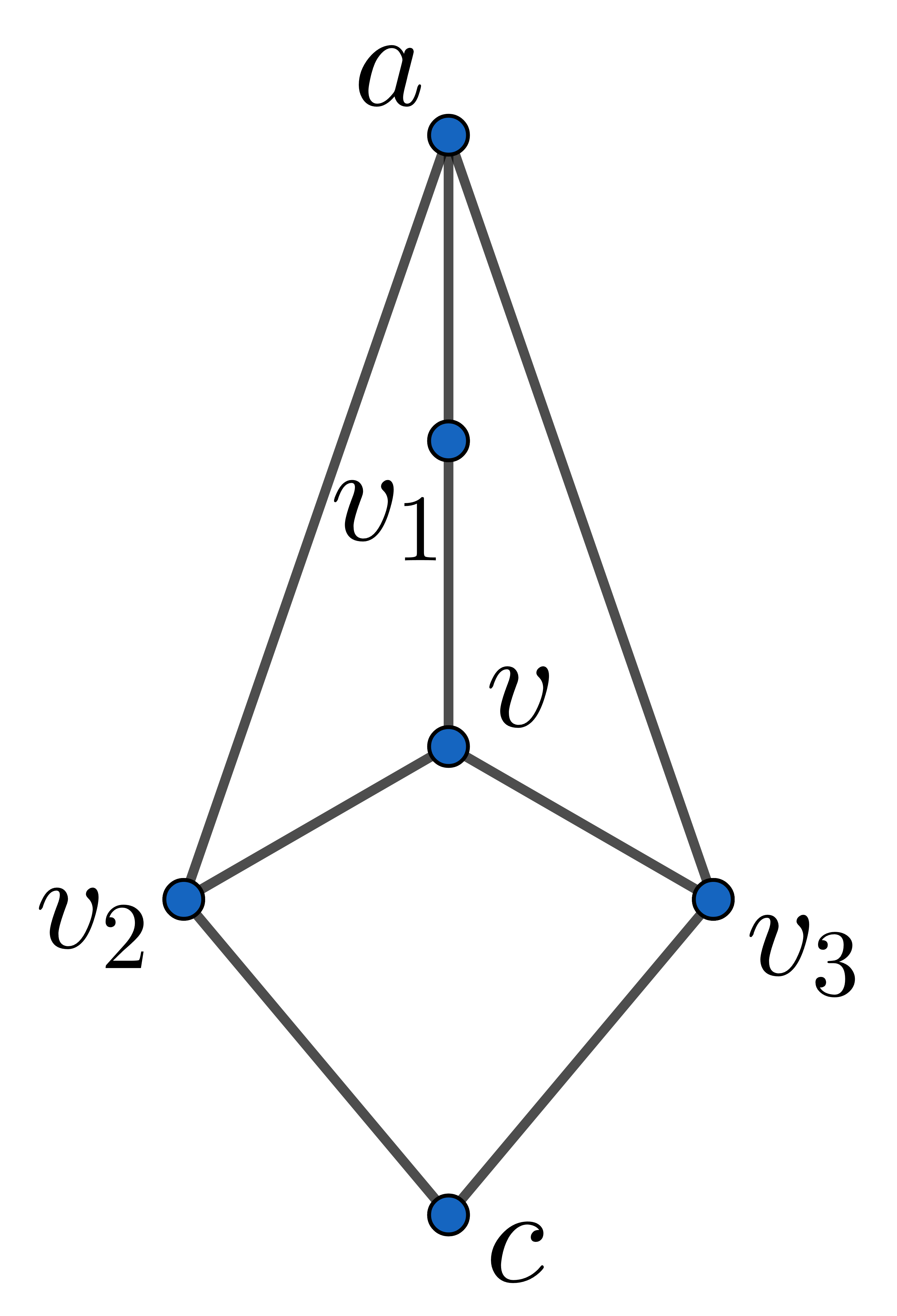}
			\caption{The case $a=b\neq c$.}
			\label{fig:pf02b}
		\end{subfigure}
		\hfill
		\begin{subfigure}{0.30\textwidth}
			\centering
			\includegraphics[width=3.5cm]{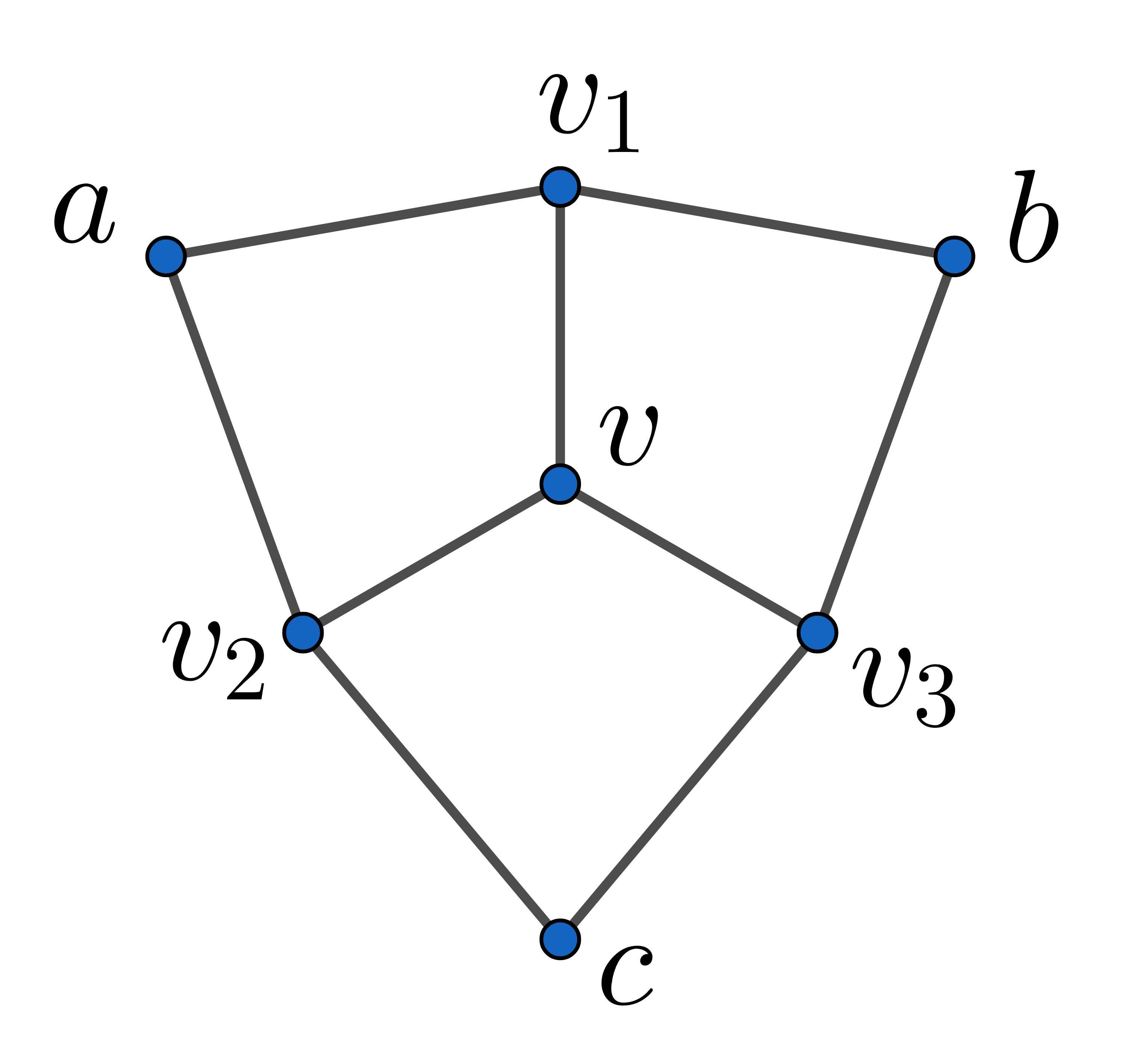}
			\caption{The case $a\neq b\neq c\neq a$.}
			\label{fig:pf02c}
		\end{subfigure}
		\caption{Proposition \ref{prop:1A}, $r=3$.}
		\label{fig:pf02}
	\end{figure}
	
	Suppose that $a=b=c$. Since $r=3$, the vertices $a,v$ have no more neighbours. By Lemma \ref{le:sc}, there can then be no vertices inside of the cycles $v,v_i,a,v_{(i\mod 3+1)}$, $i=1,2,3$. Hence $|V(G)|=5$, contradicting the handshaking lemma. Hence $a=b=c$ is impossible.
	
	Let us consider the case $a=b\neq c$. Now $v_1$ has a third neighbour $w$ other than $v,a$. By planarity, $w$ lies either inside the cycle $v,v_1,a,v_2$, or inside the cycle $v,v_1,a,v_3$. In either case, this contradicts Lemma \ref{le:sc}.
	
	The last remaining case is when $a,b,c$ are all distinct. If $a=v_3$, then $G$ is the tetrahedron, otherwise $\{v_1,v_2\}$ would be a $2$-cut. Hence if $a,b,c$ are all distinct and $G$ is not the tetrahedron, then
	\[\{a,b,c\}\cap\{v_1,v_2,v_3\}=\emptyset.\]
	Since $1\not\in A$, then $a,b$ have a common neighbour $x$ other than $v_1$, $a,c$ have a common neighbour $y$ other than $v_2$, and $b,c$ have a common neighbour $z$ other than $v_3$. Since $r=3$, we have $x=y=z$, thus $G$ is the cube.
	
	Let $r=4$, $1\not\in A$, $v\in V(G)$, and $N(v)=\{v_1,v_2,v_3,v_4\}$, in this order around $v$ in the planar immersion of $G$. Reasoning as above, let $N(v_1,v_3)\supseteq\{v,a\}$ and $N(v_2,v_4)\supseteq\{v,b\}$. By planarity, either $a=b$, or $a\in\{v_2,v_4\}$, or $b\in\{v_1,v_3\}$. If $a=b$ (Figure \ref{fig:pf05a}), then checking all cycles already drawn in $G$, by Lemma \ref{le:sc} there exist no further vertices in $G$. Hence $v_1,v_2,v_3,v_4$ form a cycle and $G$ is the octahedron.
	\begin{figure}[ht]
		\centering
		\begin{subfigure}{0.32\textwidth}
			\centering
			\includegraphics[width=3.cm]{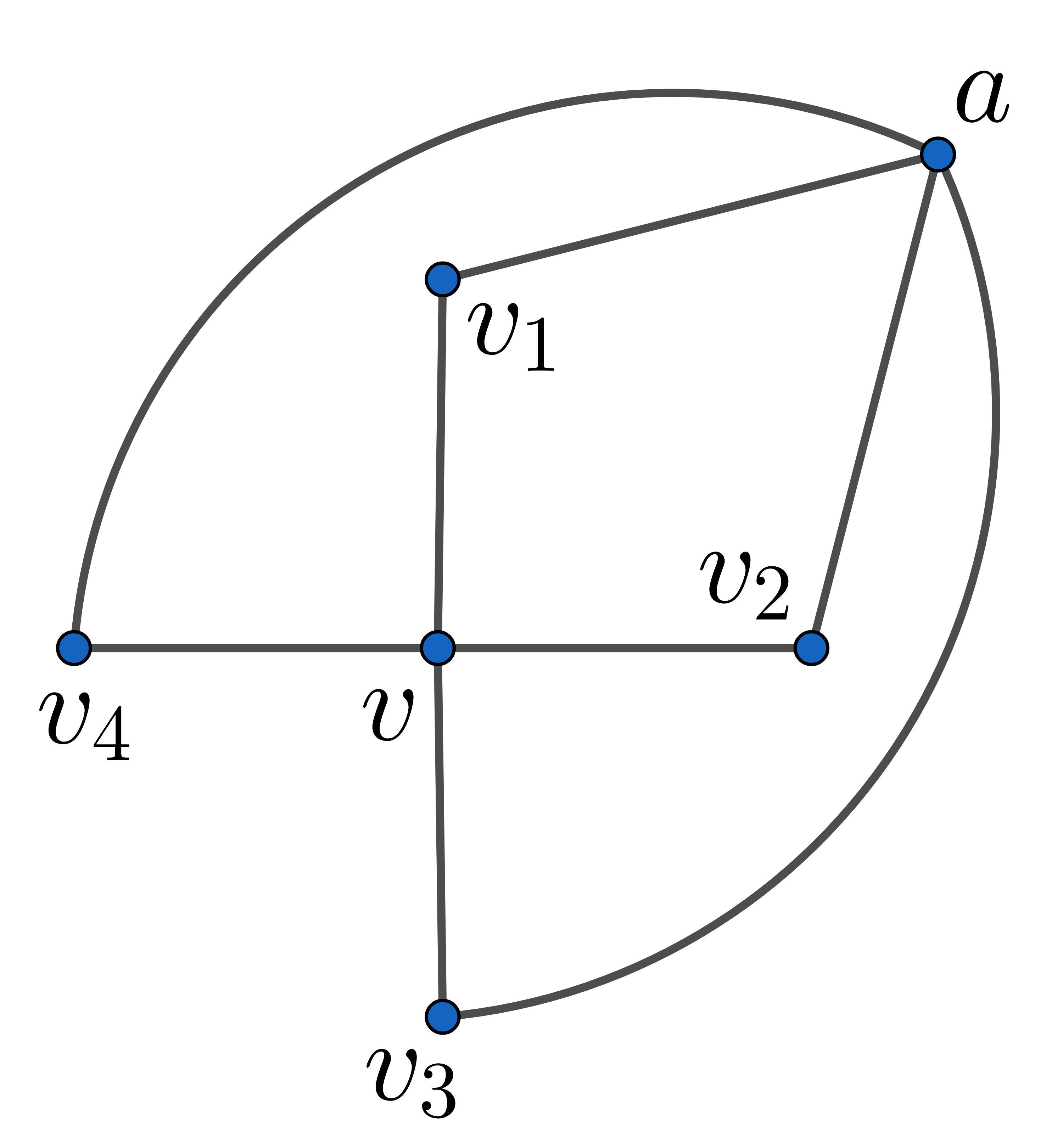}
			\caption{The case $a=b$.}
			\label{fig:pf05a}
		\end{subfigure}
		\hfill
		\begin{subfigure}{0.32\textwidth}
			\centering
			\includegraphics[width=3.25cm]{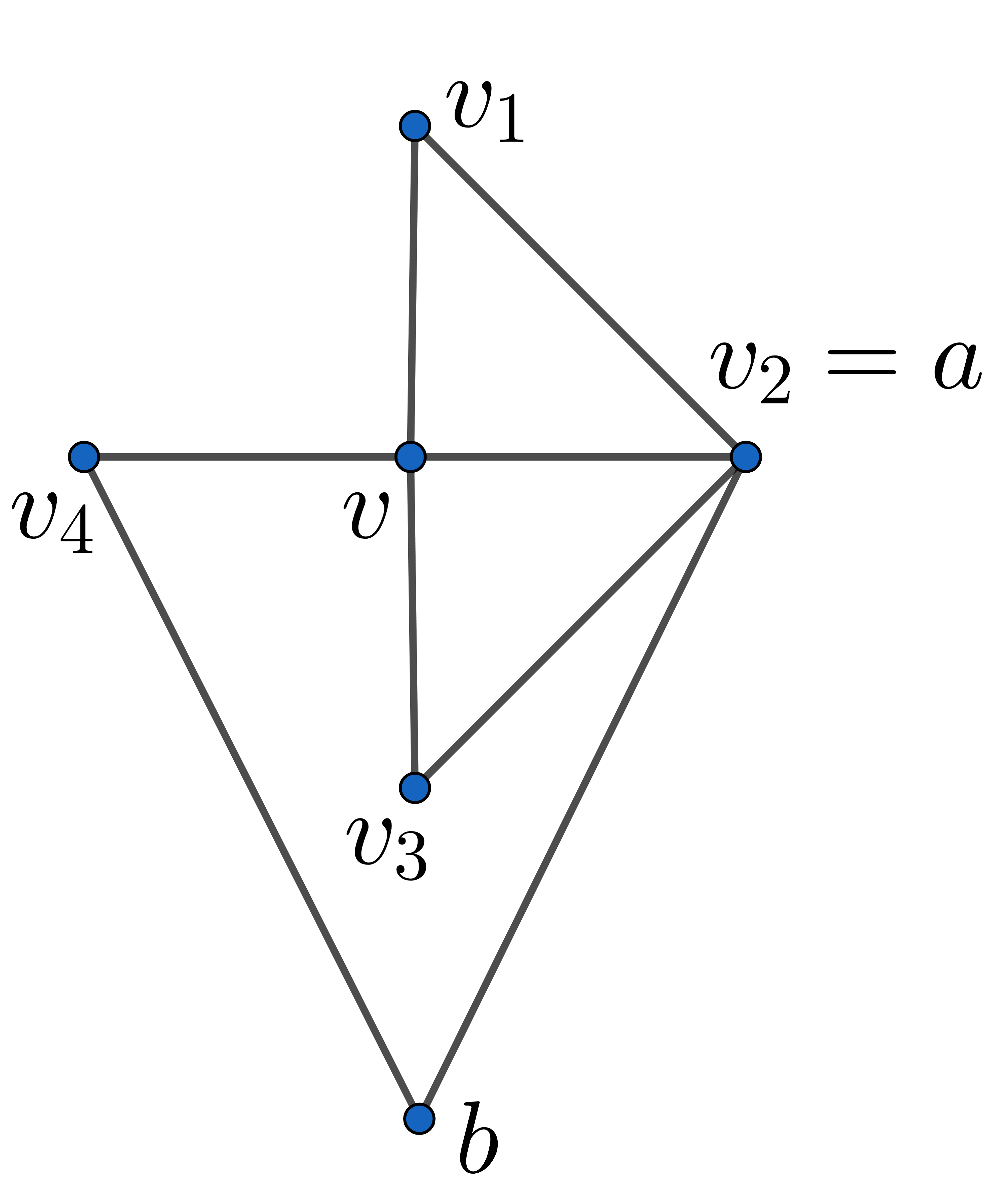}
			\caption{The case $a=v_2$ and $b\neq v_1,v_3$.}
			\label{fig:pf05b}
		\end{subfigure}
		\hfill
		\begin{subfigure}{0.32\textwidth}
			\centering
			\includegraphics[width=3.25cm]{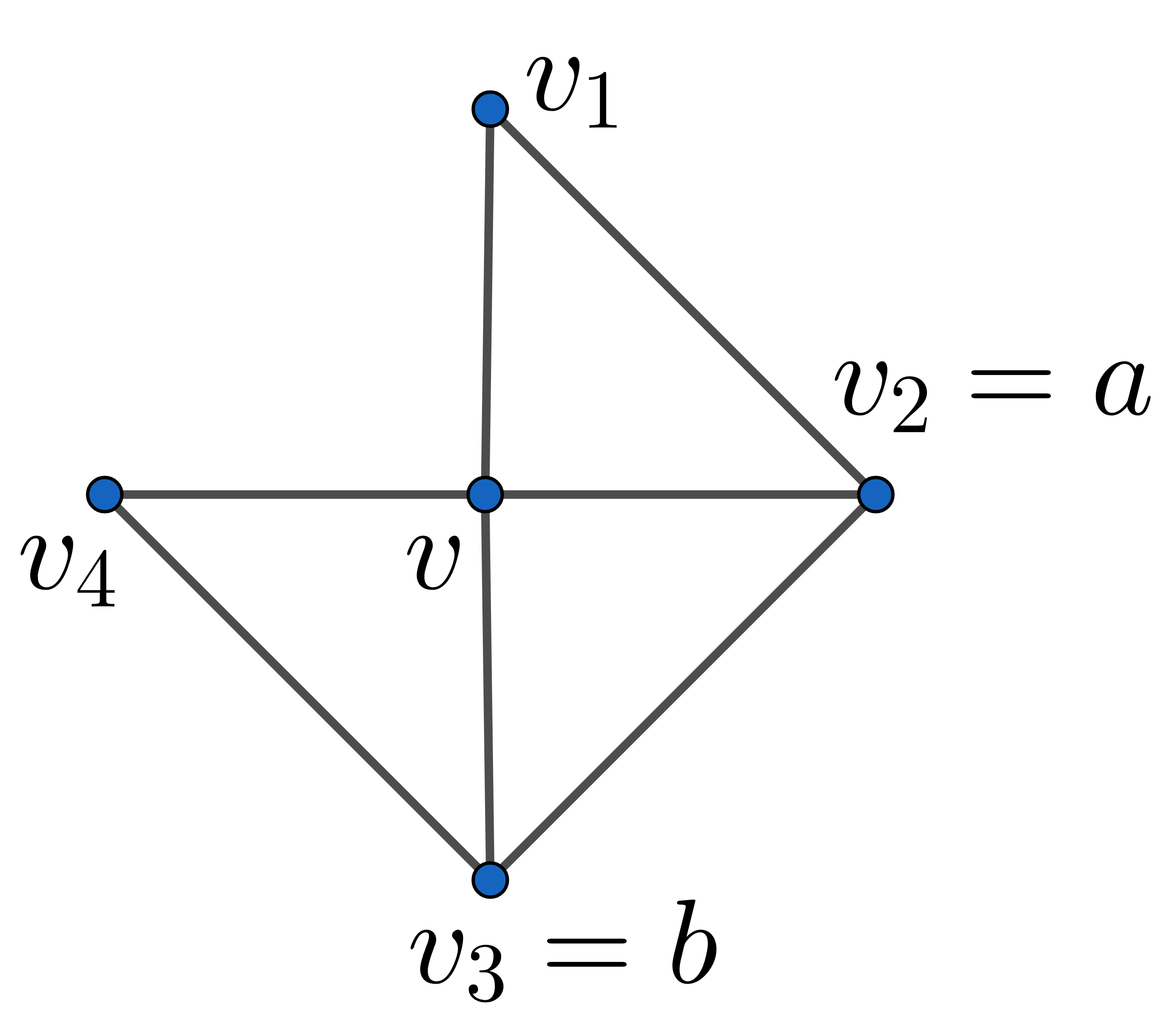}
			\caption{The case $a=v_2$ and $b=v_3$.}
			\label{fig:pf05c}
		\end{subfigure}
		\caption{Proposition \ref{prop:1A}, $r=4$.}
		\label{fig:pf05}
	\end{figure}
	
	If $a\neq b$, then w.l.o.g.\ we may assume $a=v_2$. Let $b\not\in\{v_1,v_3\}$ (Figure \ref{fig:pf05b}). By planarity, the second common neighbour of $v_1,v_2$ cannot be $v_3$, hence it is $b$. Likewise, the second common neighbour of $v_2,v_3$ is $b$. Considering the cycles $v,v_1,b,v_4$ and $v,v_3,b,v_4$, we invoke Lemma \ref{le:sc} to see that there are no further vertices in $G$, thus again $G$ is the octahedron.
	
	It remains to inspect the case $a=v_2$ and $b=v_3$ (Figure \ref{fig:pf05c}). By Lemma \ref{le:sc}, $v_1v_3\not\in E(G)$. Hence the second common neighbour of $v_1,v_2$ is a new vertex $c$. Likewise, the second common neighbour of $v_2,v_3$ is $c$. Hence the second common neighbour of $v_3,v_4$ is still $c$. Since $\{v_1,v_4\}$ is not a cutset, we have $v_1v_4\in E(G)$, thus $G$ is once more the octahedron.
	
	Let $r=5$ and $1\not\in A$. We will show that $G$ is a triangulation of the sphere, hence it is the icosahedron. By contradiction, let $F$ be a face of $G$ of length at least $4$. By planarity, we can find consecutive vertices $u_1,u_2,u_3$ on $F$ such that $u_1u_3\not\in E(G)$. In the planar immersion of $G$, the neighbours of $u_2$ in cyclic order around this vertex are
	\[u_1,a,b,c,u_3,\]
	say. Let $N(u_1,c)\supseteq\{u_2,v\}$, and $N(u_3,a)\supseteq\{u_2,w\}$. The reader may refer to Figure \ref{fig:pf08a}. Since $u_1u_3\not\in E(G)$, then by planarity we have either $v=w$, or $v=a$, or $w=c$. Now $v=w$ implies
	\[N(u_2,v)\supseteq\{u_1,u_3,a,c\},\]
	contradicting Lemma \ref{le:0123}. Hence w.l.o.g.\ we may take $v=a$, as in Figure \ref{fig:pf08b}. Now again by planarity, either $N(u_3,b)\supseteq\{u_2,a\}$, or $N(u_3,b)\supseteq\{u_2,c\}$. Since $a$ already has four neighbours, 
	%In the former case,by Lemma \ref{le:sc} the cycles $u_2,b,a,c$ and $u_2,a,b$ are facial, thus $\deg(b)=2$, contradiction.
	$N(u_3,b)\supseteq\{u_2,c\}$ (and we may take $w=c$). %Again by Lemma \ref{le:sc}, the remaining edge incident to $c$ is internal to the cycle $u_2,a,c,b$, hence 
	By planarity and since $1\not\in A$, we get $N(u_1,b)\supseteq\{u_2,a\}$ (see Figure \ref{fig:pf08c}). Now let $N(u_1,u_2)\supseteq\{a,x\}$. Since $u_1u_3\not\in E(G)$, $x\neq u_3$. By planarity, $x\neq b$. Hence $x=c$. By Lemma \ref{le:sc} the cycles $a,b,c$, $a,b,u_2$, and $c,b,u_2$ are facial. Thus $\deg(b)=3$, contradiction. Therefore, $G$ is indeed a triangulation, hence it is the icosahedron.
	\begin{figure}[ht]
		\centering
		\begin{subfigure}{0.32\textwidth}
			\centering
			\includegraphics[width=4.5cm]{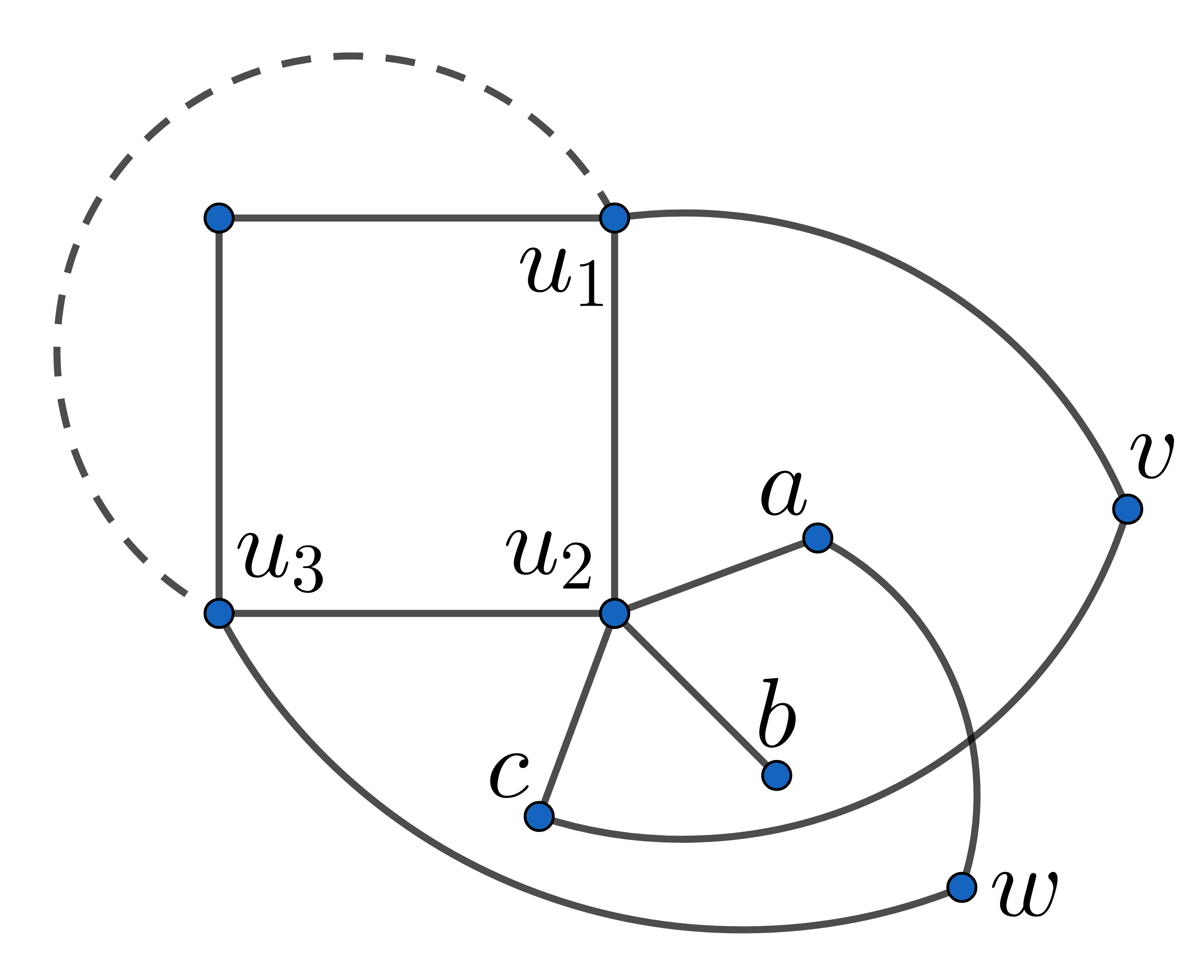}
			\caption{The dashed line indicates that $u_1u_3\not\in E(G)$.}
			\label{fig:pf08a}
		\end{subfigure}
		\hfill
		\begin{subfigure}{0.32\textwidth}
			\centering
			\includegraphics[width=4.5cm]{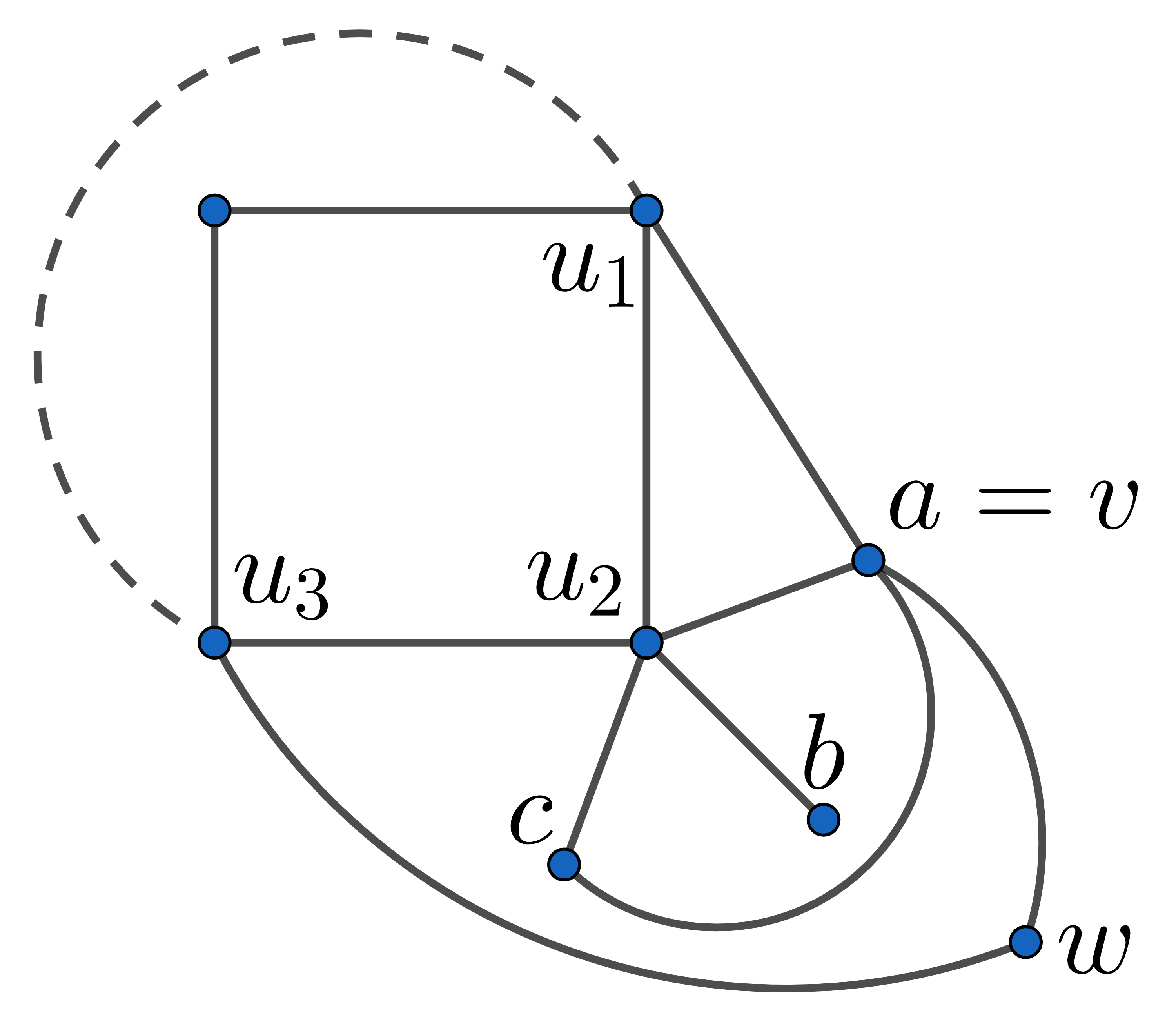}
			\caption{The case $v=a$.}
			\label{fig:pf08b}
		\end{subfigure}
		\hfill
		\begin{subfigure}{0.32\textwidth}
			\centering
			\includegraphics[width=4.75cm]{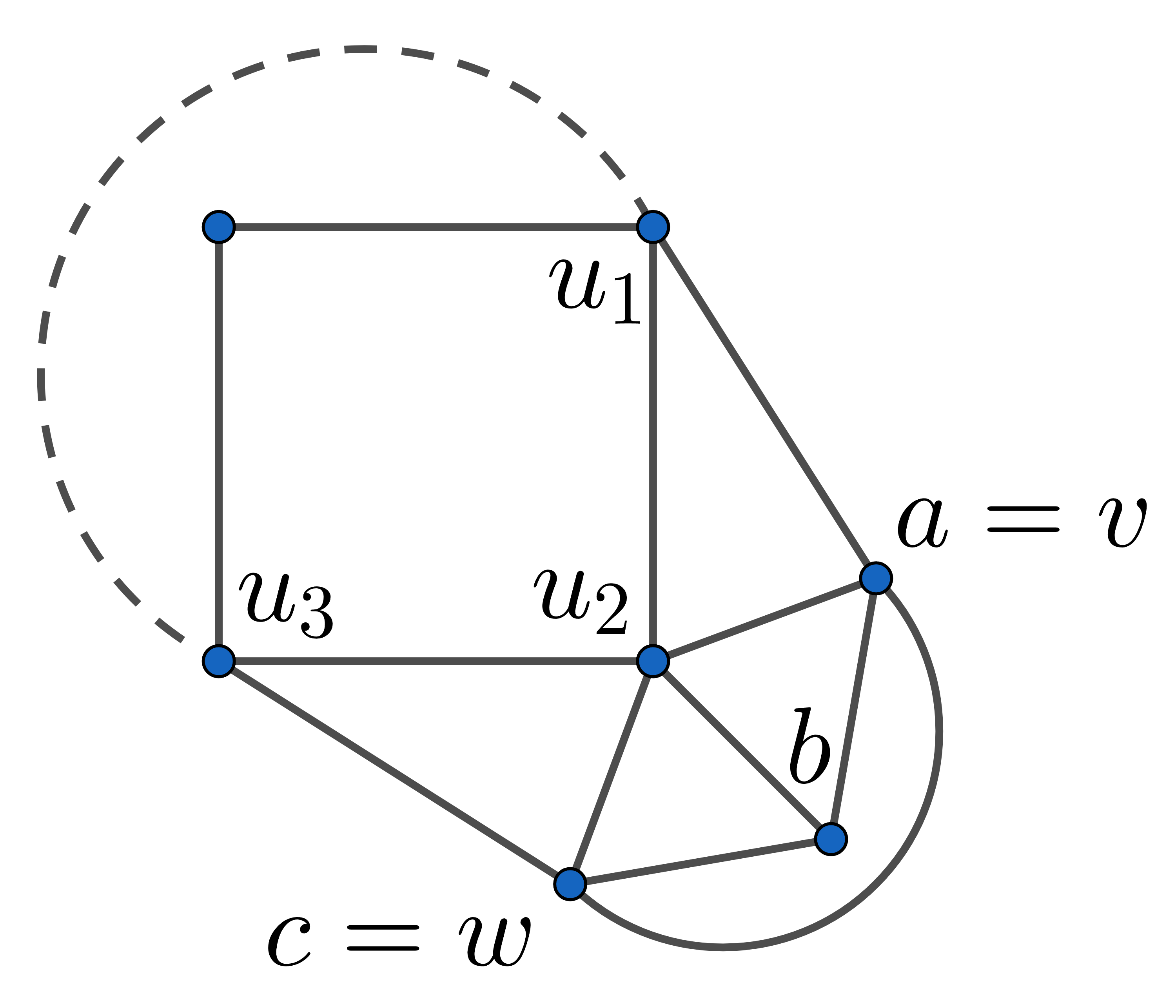}
			\caption{The case $v=a$ and $w=c$.}
			\label{fig:pf08c}
		\end{subfigure}
		\caption{Proposition \ref{prop:1A}, $r=5$.}
		\label{fig:pf08}
	\end{figure}

\paragraph{Future directions.}
It would be interesting to focus on the quintic polyhedra of the four existing types $\{0,1,2\}$, $\{0,1,2,3\}$, $\{0,1,2,4\}$, and $\{0,1,2,3,4\}$, and to discover their properties.
\\
Studying graphs via their collection of quantities of common neighbours for each pair of distinct vertices is one instance of the vast problem of $n$-degree sequences, as explained in Section \ref{sec:disc} and in \cite{maffucci2025common}. It would be intriguing to study the $n$-degree sequence of planar graphs and polyhedra.

\paragraph{Data availability statement.}
The data generated is available on kind request.

\paragraph{Declaration.}
The author has no financial or non-financial conflicts of interest to disclose.

\paragraph{Acknowledgements.}
Riccardo W. Maffucci was partially supported by Programme for Young Researchers `Rita Levi Montalcini' PGR21DPCWZ \textit{Discrete and Probabilistic Methods in Mathematics with Applications}, awarded to Riccardo W. Maffucci.
\\
The author is very grateful to Primož Šparl, and also to an anonymous referee, for helpful remarks and suggestions on a previous version.

%\clearpage
\bibliographystyle{abbrv}
\bibliography{allgra}

\begin{thebibliography}{10}

\bibitem{arcric}
D.~Archdeacon and R.~Bruce~Richter.
\newblock The construction and classification of self-dual spherical polyhedra.
\newblock {\em Journal of Combinatorial Theory, Series B}, 54(1):37--63, Jan.
  1992.

\bibitem{brin05}
G.~Brinkmann, S.~Greenberg, C.~Greenhill, B.~D. McKay, R.~Thomas, and
  P.~Wollan.
\newblock Generation of simple quadrangulations of the sphere.
\newblock {\em Discrete Mathematics}, 305(1-3):33--54, Dec. 2005.

\bibitem{brouwer2006classification}
A.~E. Brouwer.
\newblock Classification of small (0, 2)-graphs.
\newblock {\em Journal of Combinatorial Theory, Series A}, 113(8):1636--1645,
  2006.

\bibitem{brouwer1997vertex}
A.~E. Brouwer and H.~M. Mulder.
\newblock The vertex connectivity of a $\{$0, 2$\}$-graph equals its degree.
\newblock {\em Discrete Mathematics}, 169(1-3):153--155, 1997.

\bibitem{costalonga2021constructing}
J.~P. Costalonga, R.~J. Kingan, and S.~R. Kingan.
\newblock Constructing minimally 3-connected graphs.
\newblock {\em Algorithms}, 14(1):9, 2021.

\bibitem{crnkovic2025q}
D.~Crnkovi{\'c}, M.~De~Boeck, F.~Pavese, and A.~{\v{S}}vob.
\newblock q-analogs of divisible design graphs and {D}eza graphs.
\newblock {\em Journal of Combinatorial Theory, Series A}, 215:106047, 2025.

\bibitem{cui2021tight}
Q.~Cui and O.-H. Solomon~Lo.
\newblock Tight gaps in the cycle spectrum of 3-connected planar graphs.
\newblock {\em SIAM Journal on Discrete Mathematics}, 35(3):2039--2048, 2021.

\bibitem{de2024cancellation}
R.~De~March and R.~W. Maffucci.
\newblock Cancellation and regularity for planar, 3-connected {K}ronecker
  products.
\newblock {\em arXiv:2411.13473}.

\bibitem{delmaf}
J.~Delitroz and R.~W. Maffucci.
\newblock On unigraphic polyhedra with one vertex of degree $p-2$.
\newblock {\em Results in Mathematics}, 79(2):79, 2024.

\bibitem{deza1994ridge}
A.~Deza and M.~Deza.
\newblock The ridge graph of the metric polytope and some relatives.
\newblock In {\em Polytopes: Abstract, Convex and Computational}, pages
  359--372. Springer, 1994.

\bibitem{dieste}
R.~Diestel.
\newblock Graph theory 3rd ed.
\newblock {\em Graduate texts in mathematics}, 173, 2005.

\bibitem{eppstein2021polyhedral}
D.~Eppstein.
\newblock On polyhedral realization with isosceles triangles.
\newblock {\em Graphs and Combinatorics}, 37:1247--1269, 2021.

\bibitem{erickson1999deza}
M.~Erickson, S.~Fernando, W.~Haemers, D.~Hardy, and J.~Hemmeter.
\newblock Deza graphs: A generalization of strongly regular graph.
\newblock {\em Journal of Combinatorial Designs}, 7(6):395--405, 1999.

\bibitem{gaspoz2024independence}
S.~Gaspoz and R.~W. Maffucci.
\newblock Independence numbers of polyhedral graphs.
\newblock {\em Applied Mathematics and Computation}, 462:128349, 2024.

\bibitem{gavrilyuk2014vertex}
A.~Gavrilyuk, S.~Goryainov, and V.~Kabanov.
\newblock On the vertex connectivity of {D}eza graphs.
\newblock {\em Proceedings of the Steklov Institute of Mathematics},
  285:68--77, 2014.

\bibitem{gavrilyuk2024strongly}
A.~L. Gavrilyuk and V.~V. Kabanov.
\newblock Strongly regular graphs decomposable into a divisible design graph
  and a {H}offman coclique.
\newblock {\em Designs, Codes and Cryptography}, 92(5):1379--1391, 2024.

\bibitem{gengsheng2003directed}
Z.~Gengsheng and W.~Kaishun.
\newblock A directed version of {D}eza graphs-{D}eza digraphs.
\newblock {\em Australasian Journal of Combinatorics}, 28:239--244, 2003.

\bibitem{goryainov2019deza}
S.~Goryainov, W.~H. Haemers, V.~V. Kabanov, and L.~Shalaginov.
\newblock Deza graphs with parameters $(n, k, k-1, a)$ and $\beta=1$.
\newblock {\em Journal of Combinatorial Designs}, 27(3):188--202, 2019.

\bibitem{goryainov2021enumeration}
S.~Goryainov, D.~I. Panasenko, and L.~Shalaginov.
\newblock Enumeration of strictly {D}eza graphs with at most 21 vertices.
\newblock {\em Siberian Electronic Mathematical Reports}, 18(2):1423--1432,
  2021.

\bibitem{goryainov2021deza}
S.~Goryainov and L.~V. Shalaginov.
\newblock Deza graphs: a survey and new results.
\newblock {\em arXiv:2103.00228}.

\bibitem{haemers2011divisible}
W.~H. Haemers, H.~Kharaghani, and M.~A. Meulenberg.
\newblock Divisible design graphs.
\newblock {\em Journal of combinatorial theory, Series A}, 118(3):978--992,
  2011.

\bibitem{hahimo}
J.~M. Harris, J.~L. Hirst, and M.~J. Mossinghoff.
\newblock {\em Combinatorics and graph theory}, volume~2.
\newblock Springer, 2008.

\bibitem{hasheminezhad2011recursive}
M.~Hasheminezhad, B.~McKay, and T.~Reeves.
\newblock Recursive generation of simple planar 5-regular graphs and
  pentangulations.
\newblock {\em Journal of Graph Algorithms and Applications}, 15(3):417--436,
  2011.

\bibitem{hollowbread2025generation}
P.~Hollowbread-Smith and R.~W. Maffucci.
\newblock Generation of 3-connected, planar line graphs.
\newblock {\em Discrete Mathematics}, 348(2):114302, 2025.

\bibitem{jia2018generalized}
D.~Jia, L.~Yuan, and G.~Zhang.
\newblock On generalized strongly regular graphs.
\newblock {\em Graphs and Combinatorics}, 34:555--570, 2018.

\bibitem{kabanov2019strictly}
V.~V. Kabanov, N.~Maslova, and L.~Shalaginov.
\newblock On strictly {D}eza graphs with parameters $(n, k, k-1, a)$.
\newblock {\em European Journal of Combinatorics}, 80:194--202, 2019.

\bibitem{kabanov2020deza}
V.~V. Kabanov and L.~Shalaginov.
\newblock Deza graphs with parameters $(v, k, k-2, a)$.
\newblock {\em Journal of Combinatorial Designs}, 28(9):658--669, 2020.

\bibitem{limaye2005regular}
N.~Limaye, D.~Sarvate, P.~Stanica, and P.~Young.
\newblock Regular and strongly regular planar graphs.
\newblock {\em Journal of Combinatorial Mathematics and Combinatorial
  Computing}, 54:111, 2005.

\bibitem{lo2025shortness}
O.-H.~S. Lo, J.~M. Schmidt, N.~Van~Cleemput, and C.~T. Zamfirescu.
\newblock Shortness parameters of polyhedral graphs with few distinct vertex
  degrees.
\newblock {\em Discrete Mathematics}, 348(8):114518, 2025.

\bibitem{maffucci2024classification}
R.~W. Maffucci.
\newblock Classification and construction of planar, 3-connected {K}ronecker
  products.
\newblock {\em arXiv:2402.01407}.

\bibitem{maffucci2025classification}
R.~W. Maffucci.
\newblock Classification of polyhedral graphs by numbers of common neighbours.
\newblock {\em arXiv:2508.01349}.

\bibitem{maffucci2025common}
R.~W. Maffucci.
\newblock Common neighbours in planar graphs.
\newblock {\em arXiv:2511.19251}.

\bibitem{mafpo2}
R.~W. Maffucci.
\newblock Constructing certain families of $3$-polytopal graphs.
\newblock {\em Journal of Graph Theory}, pages 1--18, 2022.

\bibitem{mafpo1}
R.~W. Maffucci.
\newblock On polyhedral graphs and their complements.
\newblock {\em Aequationes Mathematicae}, pages 1--15, 2022.

\bibitem{mafpo3}
R.~W. Maffucci.
\newblock Self-dual polyhedra of given degree sequence.
\newblock {\em Art Discrete Appl. Math. 6 (2023), P1.04.}, 2023.

\bibitem{maffucci2024characterising}
R.~W. Maffucci.
\newblock Characterising 3-polytopes of radius one with unique realisation.
\newblock {\em Australasian Journal of Combinatorics}, 89(2):268--293, 2024.

\bibitem{maffucci2024rao}
R.~W. Maffucci.
\newblock Rao's theorem for forcibly planar sequences revisited.
\newblock {\em Discrete Mathematics}, 347(10):114102, 2024.

\bibitem{sedlacek1964some}
J.~Sedl{\'a}{\v{c}}ek.
\newblock Some properties of interchange graphs.
\newblock In {\em Theory of Graphs and its Applications (Proceedings of the
  Symposium, Smolenice, 1963), Publishing House of the Czechoslovak Academy of
  Sciences, Prague}, pages 145--150, 1964.

\bibitem{sedlavcek1990generalized}
J.~Sedl{\'a}{\v{c}}ek.
\newblock On generalized outerplanarity of line graphs.
\newblock {\em {\v{C}}asopis pro p{\v{e}}stov{\'a}n{\'\i} matematiky},
  115(3):273--277, 1990.

\bibitem{steinitz2013vorlesungen}
E.~Steinitz.
\newblock {\em Vorlesungen {\"u}ber die Theorie der Polyeder: unter
  Einschlu{\ss} der Elemente der Topologie}, volume~41.
\newblock Springer-Verlag, 2013.

\bibitem{wang2006deza}
K.~Wang and Y.-Q. Feng.
\newblock Deza digraphs.
\newblock {\em European Journal of Combinatorics}, 27(6):995--1004, 2006.

\bibitem{whit32}
H.~Whitney.
\newblock Congruent graphs and the connectivity of graphs.
\newblock {\em American Journal of Mathematics}, 54(1):150--168, 1932.

\bibitem{zamfirescu2023hamiltonicity}
C.~T. Zamfirescu.
\newblock On the {H}amiltonicity of a planar graph and its vertex-deleted
  subgraphs.
\newblock {\em Journal of Graph Theory}, 102(1):180--193, 2023.

\end{thebibliography}
\end{document}